\documentclass[a4paper, 11pt]{article}  
\usepackage{a4wide}
\usepackage{authblk}

\usepackage{amssymb, amsmath, amsfonts, amssymb, mathtools, amsthm}
\usepackage{hyperref}
\usepackage{xcolor}
\usepackage{caption}
\usepackage{subcaption}
\usepackage{enumitem}

\renewcommand{\leq}{\leqslant}
\renewcommand{\geq}{\geqslant}

\newcommand{\R}{\mathbb{R}}
\newcommand{\N}{\mathbb{N}}

\newcommand{\FF}{\mathcal{X}}
\newcommand{\XX}{\mathcal{X}}
\DeclareMathOperator{\Id}{Id}

\newcommand{\dd}{\mathrm{d}}

\newcommand{\YY}{\mathcal{Y}}

\newcommand{\D}{\mathcal{D}}
\newcommand{\LL}{\mathcal{L}}
\newcommand{\OO}{\mathcal{O}}
\newcommand{\II}{\mathcal{I}}

\newcommand{\what}{\hat{w}}

\newcommand{\weps}{\tilde{w}}

\newcommand{\zhat}{\hat{z}}
\newcommand{\zeps}{\tilde{z}}

\newcommand{\dbar}{\mathrm{\textnormal{d}}}

\newcommand{\zref}[1]{z_{#1, \mathrm{ref}}}
\newcommand{\zreff}{z_{\mathrm{ref}}}

\DeclareMathOperator{\Tr}{Tr}

\renewcommand{\epsilon}{\varepsilon}
\newcommand{\eps}{\epsilon}
\newcommand{\gfrak}{\mathfrak{g}}

\newcommand{\KL}{\mathcal{K}\mathcal{L}}
\newcommand{\Kinf}{\mathcal{K}_\infty}
\allowdisplaybreaks

\newtheorem{theorem}{Theorem}[section]
\newtheorem{proposition}[theorem]{Proposition}

\newtheorem{lemma}[theorem]{Lemma}

\theoremstyle{definition}
\newtheorem{definition}{Definition}[section]
\newtheorem{problem}[definition]{Problem}
\newtheorem{remark}[definition]{Remark}

\newtheorem{assumption}[definition]{Assumption}

\DeclareMathOperator{\diag}{diag}
\DeclareMathOperator{\ran}{Ran}
\DeclareMathOperator{\kernel}{Ker}
\renewcommand{\ker}{\kernel}

\usepackage{stackengine}
\newcommand\ubar[1]{\stackunder[1.2pt]{$#1$}{\rule{1.1ex}{.075ex}}}

\title{\LARGE \bf
Adaptive observer and control of spatiotemporal delayed
neural fields
}

\author[1]{Lucas Brivadis}
\author[1]{Antoine Chaillet}
\author[1]{Jean Auriol}
\affil[1]{Université Paris-Saclay, CNRS, CentraleSupélec, Laboratoire des Signaux et Systèmes, 91190, Gif-sur-Yvette, France. Emails:
        {\tt\small lucas.brivadis@centralesupelec.fr, antoine.chaillet@centralesupelec.fr, jean.auriol@centralesupelec.fr}}%

\begin{document}

\maketitle

\begin{abstract}

An adaptive observer is proposed to estimate the synaptic distribution between neurons asymptotically
from the measurement of a part of the neuronal activity and a delayed neural field evolution model.
The convergence of the observer is proved under a persistency of excitation condition.
Then, the observer is used to derive a feedback law ensuring asymptotic stabilization of the neural fields.
Finally, the feedback law is modified to ensure simultaneously practical stabilization of the neural fields and asymptotic convergence of the observer under additional restrictions on the system.
Numerical simulations confirm the relevance of the approach.

\end{abstract}



\section{Introduction}

Neural fields are nonlinear integro-differential equations used to model the activity of neuronal populations \cite{Bressloff2011,coombes2014neural}. They constitute a continuum approximation of brain structures motivated by the high density of neurons and synapses. Their infinite-dimensional nature allows for accounting for the spatial heterogeneity of the neurons' activity and the complex synaptic interconnection between them. Their delayed version also allows to take into account the non-instantaneous communication between neurons. Yet, unlike numerical models of interconnected neurons, in which every single neuron is represented by a set of differential equations, neural fields remain amenable to mathematical analysis. A vast range of mathematical tools are now available to predict and influence their behavior, including existence of stationary patterns \cite{brivadis:hal-03589737,Faugeras:2009gz} stability analysis \cite{FAYE2010561}, bifurcation analysis \cite{Atay:2004bg,Veltz:2013wm}, and feedback stabilization \cite{DECH16}.

This interesting compromise between biological significance and abstraction explains the wide range of neural fields applications, which cover primary visual cortex \cite{Bertalmio:2021uo, PINOTSIS2014143}, auditory system \cite{Boscain:2021we}, working memory \cite{Laing:2002we}, sensory cortex \cite{Detorakis:2014km}, and deep brain structures involved in Parkinson's disease \cite{DECH16}. 

The refinement of modern technologies (such as multi-electrode arrays or calcium imagining) allows to measure neuronal activity with higher and higher spatial resolution. Using these measurements to estimate the synaptic distribution between neurons would greatly help decipher the internal organization of particular brain structures. Currently, this is mostly addressed by offline algorithms based on kernel reconstruction techniques \cite{alswaihli2018kernel}, although some recent works propose online observers (see \cite{https://doi.org/10.48550/arxiv.2111.02176} for conductance-based models or \cite{brivadis:hal-03660185} for delayed neural fields). 

In turn, estimating this synaptic distribution could be of interest to improving feedback control of neuronal populations. A particularly relevant example is that of deep brain stimulation (DBS), which consists in electrically stimulating deep brain structures of the brain involved in neurological disorders such as Parkinson's disease \cite{Limousin:1998cc}. Several attempts have been made to adapt the delivered stimulation based on real-time recordings of the brain activity \cite{Carron:2013hi}. Among them, it has been shown that a stimulation proportional to the activity of a brain structure called the subthalamic nucleus is enough to disrupt Parkinsonian brain oscillations \cite{DECH16}. Yet, the value of the proportional gain depends crucially on the synaptic strength between the neurons involved: estimating it would thus allow for more respectful stimulation strategies.

In this paper, we thus develop an online strategy to estimate the synaptic distribution of delayed neural fields. This estimation relies on the assumed knowledge of the activation function of the population, the time constants involved, and the propagation delays between neurons,
as well as online measurement of a part of the neuronal activity.
It exploits the theory of adaptive observers for nonlinear systems developed in \cite{BESANCON2000271, BESANCON201715416, https://doi.org/10.48550/arxiv.2112.05497} and allows to reconstruct the unmeasured quantities based on real-time measurements. We then exploit this feature to propose a stabilizing feedback strategy that may be of particular interest to disrupt pathological brain oscillations. This control law estimates the synaptic kernel in real-time and adapts the stimulation accordingly, thus resulting in a dynamic output feedback controller.

The delayed neural fields model is presented in Section \ref{sec:preliminaries} together with an introduction to the necessary mathematical formalism. The synaptic kernel estimation is presented in Section~\ref{sec:obs}, whereas its use for feedback stabilization is presented in Section~\ref{sec:stab}. Numerical simulations to assess the performance of the proposed estimation and stabilization techniques are presented in Section \ref{sec:num}.

\section{Problem statement and mathematical preliminaries}\label{sec:preliminaries}

\subsection{Delayed neural fields}

Given a compact set $\Omega\subset \mathbb R^q$ (where, typically, $q\in\{1,2,3\}$) representing the physical support of a neuronal population, the evolution of the neuronal activity $z(t, r)\in\R^n$ at time $t\in\R_+$ and position $r\in\Omega$ is modeled as the following delayed neural fields \cite{Bressloff2011,coombes2014neural}:
\begin{align}\label{eq:wc}
    \tau(r)\frac{\partial z}{\partial t}(t, r) = &-z(t, r) + u(t, r)\nonumber \\
    &+ \int_{\Omega}w(r, r')S(z(t - d(r, r'), r'))\dd r'.
\end{align}
$n\in\N$ represents the number of considered neuronal
population types; for instance, imagery techniques often allow for discrimination between an excitatory and an inhibitory population, in which case $n=2$. $\tau(r)$ is a positive definite diagonal matrix of size $n\times n$, continuous in $r$, representing the time decay constant of neuronal activity at position $r$. $S:\R^n\to\R^n$ is a nonlinear activation function; it is often taken as a monotone function, possibly bounded (for instance, a sigmoid). $w(r, r')\in\R^{n\times n}$ defines a kernel describing the synaptic strength between locations $r$ and $r'$; its sign indicates whether the considered presynaptic neurons are excitatory or inhibitory, whereas its absolute value represents the strength of the synaptic coupling between them.
$d(r, r')\in [0, \dbar]$, for some $\dbar>0$, represents the synaptic delay between the neurons at positions $r$ and $r'$ that typically mainly results from the finite propagation speed along the axons. Finally, $u(t, r)\in\R^n$ is an input representing either the influence of non-modeled brain structures or an artificial stimulation signal. Neural fields are widely used to model neuronal populations, as reviewed in \cite{coombes2014neural,Bressloff2011}.  We stress that the synaptic kernel $w$ acts outside from the activation function $S$; in the terminology of \cite{Faugeras:2008wx}, \eqref{eq:wc} thus corresponds to a voltage-based model.


We assume that the neuronal population can be decomposed into $z(t, r)=(z_1(t, r), z_2(t, r))\in\R^{n_1}\times \R^{n_2}$ where $z_1$ corresponds to the measured part of the state and $z_2$ to the unmeasured part.
In the case where all the state is measured, we simply write $z = z_1$ and $n_2 = 0$.
Such a decomposition is natural when the two considered populations are physically separated, as it happens in the brain structures involved in Parkinson's disease \cite{DECH16}. It can also be relevant for imagery techniques that discriminate among neuron types within a given population. Accordingly, we define $\tau_i$, $S_{ij}$ $w_{ij}$ and $u_i$ of suitable dimensions for each population $i, j\in\{1,2\}$ so that
\begin{align}\label{eq:wcij}
    &\tau_i(r)\frac{\partial z_i}{\partial t}(t, r) = -z_i(t, r) + u_i(t, r)\nonumber \\
    &+ \sum_{j=1}^2\int_{\Omega}w_{ij}(r, r')S_{ij}(z_j(t-d_{ij}(r, r'), r'))\dd r'
    .
\end{align}


\subsection{Problem statement}

In the present paper, we are interested in the following control and observation problems:

\begin{problem}[Estimation]\label{pb:obs}
From the knowledge of $S_{ij}$, $w_{2j}$, $\tau_i$ and $d_{ij}$ and the online measurement of $u_i(t)$ and $z_1(t)$ for all $i, j\in\{1,2\}$,
estimate online $z_2(t)$, $w_{11}$ and $w_{12}$.
\end{problem}

\begin{problem}[Stabilization]\label{pb:stab}
From the knowledge of $S_{ij}$, $w_{2j}$, $\tau_i$ and $d_{ij}$ for all $i, j\in\{1,2\}$ and the online measurement of $z_1(t)$, find $u_1$ in the form of a dynamic output feedback law that stabilizes $z_1$ and $z_2$ at 
some reference when $u_2=0$.
\end{problem}
As already said, Problem~\ref{pb:obs} is motivated by the advances in imagery and recording technologies and the importance of determining synaptic distribution in the understanding of brain functioning. The assumption that the transmission delays are known is practically meaningful, as these delays are typically proportional to the distance $|r-r'|$ between the considered neurons via the axonal transmission speed, which is typically known a priori. Similarly, the time constants $\tau_i(r)$ are usually directly dependent on the conductance properties of the neurons. The precise knowledge of the activation function $S_{ij}$ is probably more debatable, although recent techniques allow to estimate them based on the underlying neuron type \cite{Carlu2020}.

Problem~\ref{pb:stab} is motivated by the development of deep brain stimulation (DBS) technologies that allow electrically stimulating some areas of the brain whose pathological oscillations are correlated to Parkinson's disease symptoms.
In our context, the neuronal activity measured and actuated by DBS through $u_1$ is denoted by $z_1$, which corresponds to a deep brain region known as the subthalamic nucleus (STN).
We refer to \cite{Detorakis2015} for more details on feedback techniques for DBS.
One hypothesis, defended by \cite{Holgado2010}, is that these pathological oscillations may result from the interaction between STN and a narrow part of the brain, the external globus pallidus (GPe). The neuronal activity in this area is inaccessible to measurements or stimulation in clinical practice, but it is internally stable and corresponds to $z_2$ in our model.




A strategy relying on a high-gain approach answered Problem \ref{pb:stab} in \cite{DECH16}.
The system under consideration was similar, except that the nonlinear activation function was not applied to the delayed neuronal activity but to the resulting synaptic coupling. In \cite{Faugeras:2008wx}, system~\eqref{eq:wc} is referred to as a voltage-based model, while \cite{DECH16} focused on activity-based models.
It is proven in \cite[Proposition 3]{DECH16} that under
a strong dissipativity assumption
(equivalent\footnote{Actually, there is a typo in the condition stated in \cite[Proposition 3]{DECH16}. The mistake is corrected in the proof, making it equivalent to our Assumption \ref{ass:diss}. See also \cite[Theorem 3]{DECHcdc17} for a corrected version of the hypothesis.
}
to our
Assumption \ref{ass:diss} below), for any positive continuous map $\gamma:\Omega\to\R$ and any square-integral reference signal $z_{\mathrm{ref}}:\Omega\to\R^{n_1}$ there exists a positive constant $\alpha^*$ depending on parameters of the system such that, for all $\alpha>\alpha^*$, system~\eqref{eq:wcij} coupled with the output feedback law $u_1(t, r) = -\alpha\gamma(r)(z_1(t, r)-z_{\mathrm{ref}}(r))$ and $u_2(t,r) = 0$ is globally asymptotically stable (and even input-to-state stable) at some equilibrium whose existence is proved in \cite{brivadis:hal-03589737}.
However, one of the drawbacks of this result is that $\alpha^*$ is proportional to the $L^2$-norm of $w_{11}$, which is usually unknown or uncertain.
In practice, this implies a high-gain choice in the controller,
which may lead to large values of $u_1(t, r)$ that are incompatible with the safety constraints imposed by DBS techniques.
On the contrary, our goal in this paper is to propose an adaptive strategy that does not rely on any prior knowledge of the synaptic distributions $w_{11}$ and $w_{12}$, but at the price of more knowledge on other parameters of the system.

In a preliminary work \cite{brivadis:hal-03660185}, we have shown that an observer may be designed in the delay-free case to estimate $z_2(t)$, $w_{11}$ and $w_{12}$, hence to answer Problem~\ref{pb:obs}.
However, this work was done in a framework that does not encompass time-delay systems, and Problem~\ref{pb:stab} was not addressed at all.
With such an observer, a natural dynamic output feedback stabilization strategy would be to choose
$u_1(t, r) = - \alpha (z_1(t, r) - \zref1(r)) + z_1(t, r) - \int_{r'\in\Omega}\what_{11}(t, r, r')S_{11}(z_1(t-d_{11}(r, r'), r'))\dd r' - \int_{r'\in\Omega}\what_{12}(t, r, r')S_{12}(\hat z_2(t-d_{12}(r, r'), r'))\dd r'$
where $\alpha>0$ is a tunable controller gain, 
$\zref1$ is a reference signal,
$\what_{1j}(t)$ denotes the estimation of $w_{1j}$ made by the observer at time $t$ and $\zhat_2(t)$ is the estimation of $z_2(t)$.
Doing so, if the observer has converged to the state, i.e., $\what_{1j} = w_{1j}$ and $\zhat_2 = z_2$, then the remaining dynamics of $z_1$ would be $\tau_1(r)\frac{\partial z_1}{\partial t}(t, r) = -\alpha (z_1(t, r)-\zref1(r))$ so that $z_1$ would tend towards $\zref1$. Assuming adequate contraction properties of the $z_2$ dynamics,
$z_2$ would tend towards some reference $\zref2$ that depends on $\zref1$. In particular, pathological oscillations would vanish in state-state, without any large-gain assumption on the control policy.
This motivates us to investigate Problem~\ref{pb:obs} and use the observer in dynamic output feedback to address Problem~\ref{pb:stab}.

\subsection{Definitions and notations}
Let $q$ be a positive integer, $\Omega$ be an open subset of $\R^q$ and $\FF$ be a Hilbert space endowed with the norm $\|\cdot\|_{\FF}$ and scalar product $\langle\cdot,\cdot\rangle_{\FF}$.
Denote by $L^2(\Omega, (\FF, \|\cdot\|_\FF)):= \{f:\Omega\to\FF \text{ Lebesgue-measurable}\mid \int_\Omega \|f\|_F^2<+\infty\}$
the Hilbert space of $\FF$-valued square integrable functions.
Denote by $W^{1,2}(\Omega, (\FF, \|\cdot\|_\FF)):=\{f\in L^2(\Omega, (\FF, \|\cdot\|_\FF))\mid f'\in L^2(\Omega, (\FF, \|\cdot\|_\FF)\}$
and by $W^{m,2}(\Omega, (\FF, \|\cdot\|_\FF)):=\{f\in W^{m-1,2}(\Omega, (\FF, \|\cdot\|_\FF))\mid f'\in  W^{m-1,2}(\Omega, (\FF, \|\cdot\|_\FF)\}$ for $m>1$ the usual Sobolev spaces.
If $\Omega$ is a compact set, the above definitions hold by replacing $\Omega$ by its interior, and we denote by $\mu(\Omega):=\int_\Omega \dd r$ the Lebesgue measure of $\Omega$.
If $\II$ is an interval of $\R$,
the space of $k$-times continuously differentiable functions from $\II$ to $\FF$ is denoted by $C^k(\II, \FF)$. We endow $C^0(\II, \FF)$ with the norm defined by $\|x\|_{C^0(\II, \FF)}:=\sup_{t\in\II}\|x(t)\|_\FF$ for all $x\in C^0(\II, \FF)$.
If $x\in\FF$, denote by $x^*\in\FF$ its adjoint.
If $\FF$ is a Hilbert space and $\YY$ is a Banach (resp. Hilbert) space, then the Banach (resp. Hilbert) space of linear bounded operators from $\FF$ to $\YY$ is denoted by $\LL(\FF, \YY)$. For any $W\in\LL(\FF, \YY)$, denote by $\ker W$ its kernel and $\ran W$ its range.
The map $\FF\ni x\mapsto \|Wx\|_\YY$ defines a semi-norm on $W$, that is said to be induced by $W$. It is a norm if and only if $W$ is injective.
Set $\LL(\FF):= \LL(\FF, \FF)$.
Denote by $\Id_\FF$ the identity operator over $\FF$.

For any positive integers $n$ and $m$ and any matrix $w\in\R^{n\times m}$, denote by $w^\top$ its transpose, $\Tr(w)$ its trace, $\|w\|$ its norm induced by the Euclidean norm, and $\|w\|_{F}=\sqrt{\Tr(w^\top w)}$ its Frobenius norm. Recall that these norms are equivalent and $\|w\|\leq \|w\|_{F}$. Hence, for any positive integers $n$ and $m$, $L^2(\Omega, (\R^{n\times m}, \|\cdot\|))$ and $L^2(\Omega, (\R^{n\times m}, \|\cdot\|_F))$ are equivalent Hilbert spaces and $\|\cdot\|_{L^2(\Omega, (\R^{n\times m}, \|\cdot\|))} \leq \|\cdot\|_{L^2(\Omega, (\R^{n\times m}, \|\cdot\|_F))}$.

For all $i, j\in\{1,2\}$,
set $\FF_{z_i} = L^2(\Omega, \R^{n_i})$ and $\FF_{w_{ij}} = L^2(\Omega^2, \R^{n_i\times n_j})$, so that $\FF_{z_i}$ (resp. $\FF_{w_{ij}}$) will be used as the state space of $z_i$ (resp. $w_{ij}$).
Set also $\FF_z = \FF_{z_1}\times \FF_{z_2}$ and $\FF_w = \FF_{w_{11}}\times \FF_{w_{12}}$.
By abuse of notations, we write $\|\cdot\|_{(\FF_{w_{ij}}, \|\cdot\|)} := \|\cdot\|_{L^2(\Omega^2, (\R^{{n_i}\times {n_j}}, \|\cdot\|))}$ and $\|\cdot\|_{(\FF_{w_{ij}}, \|\cdot\|_F)} := \|\cdot\|_{L^2(\Omega^2, (\R^{{n_i}\times {n_j}}, \|\cdot\|_F))}$.
For any positive constant $\dbar$ and any Hilbert space $\FF$, if $x\in C^0([-\dbar, +\infty), \FF)$, we denote by $x_t$ the history of $x$ over the latest time interval of length $\dbar$, i.e., $x_t(s) = x(t+s)$ for all $t\geq0$ and all $s\in[-\dbar, 0]$. 
For any $n\in\N$, denote by $\D^n_{++}\subset\R^{n\times n}$ the set of positive diagonal matrices.
For the sake of reading, if $d_{i2}$ does not depend on $r$, i.e. if $d_{i2}(\cdot, r')$ is constant for all $r'\in\Omega$, we simply write $d_{ij}(r') := d_{ij}(r, r')$. For any globally Lipschitz map $S_{ij}:\R^{n_j}\to\R^{n_i}$,
denote by $\ell_{ij}$ its Lipschitz constant and set $\bar S_{ij}:=\sup_{\R^{n_i}}|S_{ij}|$ .

\begin{remark}
The set $\Omega$ is defined as a compact set of $\mathbb R^q$ where $q$ is typically an integer with values 1, 2, or 3, depending on the considered dimension of the neuronal population. For instance, for neuronal phenomena involving a single direction, we can consider $q=1$ \cite{ermentrout1993existence}. For in vitro cultures, brain slices, or to study planar patterns, we can often consider $q=2$ \cite{tamekue2023mathematical}, whereas for in vivo populations, in order to model the volume of a brain structure, we will typically consider $q=3$ \cite{DECHPASE15}. No specific mathematical form is imposed for $\Omega$, as long as it is compact. It can be made of a connected set or be the union of non-overlapping sets to model physically distinct brain structures \cite{DECHPASE15}.
\end{remark}

\begin{remark}\label{rem:mes}
To ease the reading, we have chosen to consider that the integral over $\Omega$ is a Lebesgue integral, i.e., that $\Omega$ is endowed with the Lebesgue measure. However, note that our work remains identical when considering any other measure for which $\Omega$ is measurable. In particular, an interesting case is when $\Omega = \cup_{k=1}^N\{r_k\}$ for some finite family $(r_k)_{1\leq k\leq N}$ in $\R^q$ and the measure is the counting measure. In that case, \eqref{eq:wc} can be rewritten as the usual finite-dimensional Wilson-Cowan equation \cite{wilson1973mathematical}: for all $k \in \{1,\dots,N\}$,
\begin{align}\label{eq:wc_discrete}
    \tau(r_k)&\frac{\partial z}{\partial t}(t, r_k) = -z(t, r_k) + u(t, r_k) \nonumber \\&+\sum_{\ell=1}^Nw(r_k, r_\ell)S(z(t - d(r_k, r_\ell), r_\ell).
\end{align}
This case will be further investigated in Section~\ref{sec:sim}.
\end{remark}

\subsection{Preliminaries on Hilbert-Schmidt operators}

Let $q$, $n$ and $m$ be positive integers and $\Omega$ be an open subset of $\R^q$.
To any map $w\in L^2(\Omega^2, \R^{n\times m})$, one can associates a Hilbert-Schmidt (HS) integral operator $W:L^2(\Omega, \R^{m})\to L^2(\Omega, \R^{n})$ defined by $(Wz)(r) = \int_\Omega w(r, r')z(r')\dd r'$ for all $r\in\Omega$. The map $w$ is said to be the kernel of $W$. Let us recall some basic notions on such operators (see e.g. \cite{Gohberg1990} for more details). The space $\LL_2(L^2(\Omega, \R^{m}), L^2(\Omega, \R^{n}))$ of HS integral operators is a subspace of $\LL(L^2(\Omega, \R^{m}),$ $ L^2(\Omega, \R^{n}))$, and is a Hilbert space when endowed with the scalar product defined by
$$
\langle W_a, W_b\rangle_{\LL_2(L^2(\Omega, \R^{m}), L^2(\Omega, \R^{n}))} := \langle w_a, w_b\rangle_{(L^2(\Omega^2, \R^{n\times m}, \|\cdot\|_F))}
$$
for all $W_a$ and $W_b$ in $\LL_2(L^2(\Omega, \R^{n\times m}))$ with kernels $w_a$ and $w_b$, respectively.
For any Hilbert basis $(e_k)_{k\in\N}$ of $L^2(\Omega, \R^{m})$, we have that
$\|W\|^2_{\LL_2(L^2(\Omega, \R^{m}), L^2(\Omega, \R^{n}))} = \sum_{k\in\N}\|We_k\|^2_{L^2(\Omega, \R^{n})}$
for all $W\in \LL_2(L^2(\Omega, \R^{m}), L^2(\Omega, \R^{n}))$, i.e.,
\begin{equation*}
\|w\|_{L^2(\Omega^2, (\R^{n\times m}, \|\cdot\|_F))}^2 = \sum_{k\in\N} \int_\Omega \Big|\int_\Omega w(r, r')e_k(r')\dd r'\Big|^2\dd r
\end{equation*}
for all $w\in L^2(\Omega^2, \R^{n\times m})$.

Let $p$ be another positive integer.
If $W$ and $P$ are two HS integral operators with kernels $w$ and $\rho$, in $\LL_2(L^2(\Omega, \R^{m}), L^2(\Omega, \R^{n}))$ and $\LL_2(L^2(\Omega, \R^{n}), L^2(\Omega, \R^{p}))$ respectively, the composition $WP$ is also a HS integral operators, in $\LL_2(L^2(\Omega, \R^{m}),$ $L^2(\Omega, \R^{p}))$. Moreover, its kernel is denoted by $w\circ\rho$ and satisfies
$$(w\circ\rho)(r, r') = \int_\Omega w(r, r'')\rho(r'', r')\dd r''$$
for all 
$r, r'\in\Omega$.

If $W\in \LL_2(L^2(\Omega, \R^{m}), L^2(\Omega, \R^{n}))$ has kernel $w$, then its adjoint $W^*$ is also a HS integral operator and its kernel $w^*$ satisfies $w^*(r, r') = w(r', r)^\top$ for all $r, r'\in\Omega$. In particular, $W$ is self-adjoint if and only if $n=m$ and $w(r, r') = w(r', r)^\top$ for all $r, r'\in\Omega$, and $w$ is a positive-definite kernel if and only if so is $W$. In that case, $w$ induces a norm on $L^2(\Omega, \R^n)$, defined by $z\mapsto\|W z\|_{L^2(\Omega, \R^n)}$, that is weaker than or equivalent to the usual norm $\|\cdot \|_{L^2(\Omega, \R^n)}$.

To answer Problem~\ref{pb:obs}, we propose to estimate $w_{1j}$'s in the norm $L^2(\Omega, (\R^{n\times m},$ $\|\cdot\|_F))$, which, by the previous remarks, is equivalent to estimate their associated HS operators.
This operator-based approach has been followed in \cite{brivadis:hal-03660185} to answer Problem~\ref{pb:obs} in the delay-free case. In the present paper, we focus on estimating the kernels rather than their associated operators. From a practical viewpoint, the use of the Frobenius norm corresponds to a coefficientwise estimation of the matrices $w_{1j}(r, r')$.

\subsection{Properties of the system}

Let us recall the well-posedness of the system \eqref{eq:wcij} under consideration (that was proved in \cite[Theorem 3.2.1]{FAYE2010561}), as well as a bounded-input bounded-state (BIBS) property (that we prove below).

\begin{assumption}\label{ass:wp}
The set $\Omega\subset\R^q$ is compact and,
for all $i, j\in\{1, 2\}$,
$\tau_i\in C^0(\Omega, \D^{n_i}_{++})$, 
$u_i\in C^0(\R_+, \FF_{z_i})$, 
$d_{ij}\in C^0(\Omega^2, [0, \dbar])$ for some $\dbar>0$,
$w_{ij}\in \FF_{w_{ij}}$,
and $S_{ij}\in C^0(\R^{n_j}, \R^{n_i})$ is bounded and globally Lipschitz.
\end{assumption}
These assumptions are standard in neural field analysis (see, e.g., \cite{DECH16}). In particular, the boundedness of $S$ reflects the biological limitations of the maximal activity that the population can reach.

\begin{proposition}[Open-loop well-posedness and BIBS]\label{th:bibs}
Suppose that Assumption~\ref{ass:wp} is satisfied.
Then, for any initial condition $(z_{1,0}, z_{2,0})\in C^0([-\dbar, 0], \FF_{z_1})\times C^0([-\dbar, 0], \FF_{z_2})$,
the open-loop system \eqref{eq:wcij} admits a unique corresponding solution $(z_1, z_2)\in C^1([0, +\infty), \FF_{z_1}\times\FF_{z_2})\cap C^0([-\dbar, +\infty), \FF_{z_1}\times\FF_{z_2})$.
Moreover, if $u_i$ is bounded for all $i\in\{1,2\}$,
then all solutions $(z_1, z_2)$ of \eqref{eq:wcij}
are such that $z_i$ and $\frac{\dd z_i}{\dd t}$ are also bounded.
\end{proposition}

\begin{proof}\emph{Well-posedness.}
The only difference with \cite[Theorem 3.2.1]{FAYE2010561} is that $\tau$ depends on $r$. However, since $\tau$ is assumed to be continuous and positive, the proof remains identical to the one given in \cite[Theorem 3.2.1]{FAYE2010561}.

\emph{BIBS.}
    For all $i\in\{1,2\}$, let $\ubar{\tau}_i$ be the smallest diagonal entry of $\tau_i(r)$ when $r$ spans $\Omega$
    (which exists since $\tau_i$ is continuous and $\Omega$ is compact).
    For all $t\geq0$, we have by Young's and Cauchy-Schwartz inequalities that
    \begin{align*}
       & \frac{\ubar{\tau}_i}{2}\frac{\dd}{\dd t} \|z_i(t)\|_{\FF_{z_i}}^2
        \leq - \int_\Omega |z_i(t, r)|^2 \dd r
        + \int_\Omega z_i(t, r)^\top u_i(t, r) \dd r
        \\
        &\quad+ \int_{\Omega} z_i(t, r)^\top \sum_{j=1}^2 \int_{\Omega} w_{ij}(r, r') S_{ij}(z_j(t-d_{ij}(r, r'), r')) \dd r'\dd r
        \\
        &\leq - \int_\Omega |z_i(t, r)|^2 \dd r +\frac{1}{4} \int_\Omega |z_i(t, r)|^2 \dd r + \int_\Omega |u_i(t, r)|^2 \dd r
        \\
        &\quad + \frac{1}{4} \int_\Omega |z_i(t, r)|^2 \dd r 
        \\
        &\quad + \sum_{j=1}^2 \int_{\Omega^2} \|w_{ij}(r, r')\|^2 |S_{ij}(z_j(t-d_{ij}(r, r'), r'))|^2 \dd r'\dd r
        \\
        &\leq - \frac{1}{2} \|z_i(t)\|_{\FF_{z_i}}^2 + \|u_i(t)\|_{\FF_{u_i}}^2 + \sum_{j=1}^2\bar S_{ij}^2\|w_{ij}\|^2_{(\FF_{w_{ij}}, \|\cdot\|)}.
    \end{align*}
Hence, if $u_i$ remains bounded, then $z_i$ also remains bounded by Grönwall's inequality.
Moreover,
\begin{align*}
    &\|\tau_i\frac{\partial z_i}{\partial t}(t)\|_{\FF_{z_i}} 
\leq  \|z_i(t)\|_{\FF_{z_i}} + \|u_i(t)\|_{\FF_{z_i}} 
\\
&\quad + \sqrt{\int_\Omega \Bigg|\sum_{j=1}^2\int_{\Omega}w_{ij}(r, r')S_{ij}(z_j(t-d_{ij}(r, r'), r'))\dd r'\Bigg|^2\dd r}
\\
&\leq  \|z_i(t)\|_{\FF_{z_i}} + \|u_i(t)\|_{\FF_{z_i}} + \sum_{j=1}^2\bar S_{ij}\|w_{ij}\|_{(\FF_{w_{ij}}, \|\cdot\|)}.
\end{align*}
Hence $\frac{\dd z_i}{\dd t}$ is also bounded if $u_i$ is bounded.
\end{proof}

In the rest of the paper, we always make the Assumption~\ref{ass:wp}, so that the well-posedness of the system is always guaranteed.

\section{Adaptive observer}\label{sec:obs}

\subsection{Observer design}

In order to design an observer, we first make a dissipativity assumption on the unmeasured part $z_2$ of the state.


\begin{assumption}[Strong dissipativity]\label{ass:diss}
It holds that
$\ell_{22} \|w_{22}\|_{(\FF_{w_{22}}, \|\cdot\|)}<1$.
\end{assumption}
Assumption~\ref{ass:diss} yields that for any pair $(z_2^a, z_2^b)$ of solutions of \eqref{eq:wc} (replacing $\tau$, $w$, $S$ and $d$ by $\tau_{22}$, $w_{22}$, $S_{22}$ and $d_{22}$), the distance $\|z_2^a(t) - z_2^b(t)\|_{\FF_{z_2}}$ is converging towards $0$ as $t$ goes to $+\infty$ (this fact will be proved and explained in Remark~\ref{rem:diss}).
Assumption~\ref{ass:diss} can thus be interpreted as a detectability hypothesis: the unknown part of the state has contracting dynamics with respect to some norm.


We also stress that Assumption~\ref{ass:diss} is commonly used in the stability analysis of neural fields \cite{Faugeras:2008wx} and ensures dissipativity even in the presence of axonal propagation delays \cite{DECHcdc17}.

Inspired by the delay-free case investigated in \cite{brivadis:hal-03660185}, let us consider the following observer:

\begin{equation}\label{eq:obs}
\begin{aligned}
&\begin{aligned}
    \tau_1(r)\frac{\partial \zhat_1}{\partial t}(t, r) =& -\alpha(\zhat_1(t, r)-z_1(t, r)) - z_1(t, r) + u_1(t, r)
    \\
    &
    +\int_{\Omega}\what_{11}(t, r, r')S_{11}(z_1(t-d_{11}(r, r'), r'))\dd r'
    \\
    &
    +\int_{\Omega}\what_{12}(t, r, r')S_{12}(\zhat_2(t-d_{12}(r, r'), r'))\dd r'
\end{aligned}
\\
&\begin{aligned}
    \tau_2(r)\frac{\partial \zhat_2}{\partial t}(t, r) =& -\zhat_2(t, r) + u_2(t, r)
    \\
    &+\int_{\Omega}w_{21}(r, r')S_{21}(z_1(t-d_{21}(r, r'), r'))\dd r'
    \\
    &
    +\int_{\Omega}w_{22}(r, r')S_{22}(\zhat_2(t-d_{22}(r, r'), r'))\dd r'
\end{aligned}
\\
&\begin{aligned}
    \tau_1(r)\frac{\partial \what_{11}}{\partial t}(t, r, r') =& -(\zhat_1(t, r)-z_1(t, r))\\
    & S_{11}(z_1(t-d_{11}(r, r'), r'))^\top
\\
\tau_1(r)\frac{\partial \what_{12}}{\partial t}(t, r, r') =& -(\zhat_1(t, r)-z_1(t, r))\\
&\quad \quad S_{12}(\zhat_2(t-d_{12}(r, r'), r'))^\top
\end{aligned}
\end{aligned}
\end{equation}
where $\alpha>0$ is a tunable observer gain, to be selected later.

Note that $\zhat_2$ has the same dynamics as $z_2$. Hence the dissipativity Assumption~\ref{ass:diss} shall be employed to prove observer convergence. The correction terms are inspired by \cite{BESANCON2000271} that dealt with the finite-dimensional delay-free context.

The well-posedness of the observer system is a direct adaptation of \cite[Theorem 3.2.1]{FAYE2010561}. The main differences are that $\tau_i$'s are space-dependent and $\what_{1i}$ are solutions of a dynamical system.

\begin{proposition}[Observer well-posedness]\label{th:wp}
Suppose that Assumption~\ref{ass:wp} is satisfied.
Then, for any initial condition $(z_{1,0}, \zhat_{1,0}, z_{2,0}, \zhat_{2,0}, \what_{11,0}, \what_{12,0})\in$ $C^0($ $[-\dbar, 0],$ $ \FF_{z_1})^2\times C^0([-\dbar, 0], \FF_{z_2})^2\times\FF_{w_{11}}\times\FF_{w_{12}}$,
the open-loop system \eqref{eq:wcij}-\eqref{eq:obs} admits a unique corresponding solution $(z_1, \hat z_1, z_2, \zhat_2, \what_{11}, \what_{12})\in C^1([0, +\infty), \FF_{z_1}^2\times\FF_{z_2}^2\times\FF_{w_{11}}\times\FF_{w_{12}})\cap C^0([-\dbar, +\infty), \FF_{z_1}^2\times\FF_{z_2}^2\times\FF_{w_{11}}\times\FF_{w_{12}})$.
\end{proposition}
\begin{proof}
The proof is based on \cite[Lemma 2.1 and Theorem 2.3]{hale2013introduction}, and follows the lines of \cite[Lemma 3.1.1]{FAYE2010561}.
First, note that \eqref{eq:wcij}-\eqref{eq:obs} is a cascade system where the observer \eqref{eq:obs} is driven by the system's dynamics \eqref{eq:wcij}.
The well-posedness of \eqref{eq:wcij} is guaranteed by Proposition~\ref{th:bibs}. Now, let $(z_1, z_2)$ be a solution of \eqref{eq:wcij} and let us prove the existence and uniqueness of $(\zhat_1, \zhat_2,\what_{11},\what_{12})$ solution of \eqref{eq:obs} starting from the given initial condition.
Let us consider the map $F:\R_+ \times \FF_{z_1}\times C^0([-\dbar, 0], \FF_{z_2})\times\FF_{w_{11}}\times\FF_{w_{12}}
\to
\FF_{z_1}\times  \FF_{z_2}\times\FF_{w_{11}}\times\FF_{w_{12}}
$
such that \eqref{eq:obs} can be rewritten as
$\frac{\dd}{\dd t}(\zhat_{1}, \zhat_2,\what_{11},\what_{12})(t) = F(t, \zhat_{1}, \zhat_{2t},\what_{11},\what_{12})$.
Since $\tau_i$ are continuous and positive, $S_{ij}$ are bounded and $w_{ij}$ are square-integrable over $\Omega^2$, $d_i$ are continuous and $u_i\in C^0(\R_+, \FF_{z_i})$, the map $F$ is well-defined by the same arguments than \cite[Lemma 3.1.1]{FAYE2010561}.
Let us show that $F$ is continuous, and globally Lipschitz with respect to $(\zhat_{1}, \zhat_{1t},\what_{11},\what_{12})$, so that we can conclude with \cite[Lemma 2.1 and Theorem 2.3]{hale2013introduction}.
Define $F_1$ taking values in $\FF_{z_1}$, $F_2$ taking values in $\FF_{z_2}$, $F_3$ taking values in $\FF_{w_{11}}$ and $F_4$ taking values in $\FF_{w_{12}}$ so that $F = (F_i)_{i\in\{1,2,3,4\}}$
From the proof of \cite[Lemma 3.1.1]{FAYE2010561}, $F_1$ and $F_2$ are continuous and globally Lipschitz with respect to the last variables.
From the boundedness of $S$, $F_3$ and $F_4$ are also continuous and globally Lipschitz with respect to the last variables.
This concludes the proof of Proposition~\ref{th:wp}.
\end{proof}
Let us define the estimation error $(\zeps_1, \zeps_2, \weps_{11}, \weps_{12}) = (\zhat_1-z_1, \zhat_2-z_2, \what_{11}-w_{11}, \what_{12}-w_{12})$.
It is ruled by the following dynamical system:
\begin{equation}\label{eq:eps}
\begin{aligned}
&\begin{aligned}
    &\tau_1(r)\frac{\partial \zeps_1}{\partial t}(t, r)
    = -\alpha\zeps_1(t, r)
    +\int_{\Omega}\weps_{11}(t, r, r')S_{11}(z_1(t- d_{11}
    \\
    &
    (r, r'), r'))\dd r'+\int_{\Omega}\what_{12}(t, r, r')S_{12}(\zhat_2(t-d_{12}(r, r'), r'))\dd r'
    \\
    &
    \quad-
    \int_{\Omega}w_{12}(r, r')S_{12}(z_2(t-d_{12}(r, r'), r'))\dd r'
    \\
    &
    = -\alpha\zeps_1(t, r)
    +\int_{\Omega}\weps_{11}(t, r, r')S_{11}(z_1(t-d_{11}(r, r'), r'))\dd r'
    \\
    &
    \quad+\int_{\Omega}\weps_{12}(t, r, r')S_{12}(\zhat_2(t-d_{12}(r, r'), r'))\dd r'
    \\
    &
    \quad+\int_{\Omega}w_{12}(r, r')(S_{12}(\zhat_2(t-d_{12}(r, r'), r'))\nonumber\\
    &\quad -S_{12}(z_2(t-d_{12}(r, r'), r')))\dd r'
    \end{aligned}
\\
&\begin{aligned}
    \tau_2(r)\frac{\partial \zeps_2}{\partial t}(t, r) &= -\zeps_2(t, r)+\int_{\Omega}w_{22}(t, r, r')(S_{22}(\zhat_2(t-
    \\&
    d_{22}(r, r'), r'))-S_{22}(z_2(t-d_{22}(r, r'), r')))\dd r'
\end{aligned}
\\
&
    \tau_1(r)\frac{\partial \weps_{11}}{\partial t}(t, r, r') = -\zeps_1(t, r)S_{11}(z_1(t-d_{11}(r, r'), r'))^\top
\\
&
    \tau_2(r)\frac{\partial \weps_{12}}{\partial t}(t, r, r') = -\zeps_1(t, r)S_{12}(\zhat_2(t-d_{12}(r, r'), r'))^\top
\end{aligned}
\end{equation}

\subsection{Observer convergence}

In what follows, we wish to exhibit sufficient conditions for the convergence of the observer towards the state, meaning the convergence of the estimation error $(\zeps_1, \zeps_2, \weps_{11}, \weps_{12})$ towards $0$. To do so, we introduce a notion of persistence of excitation over infinite-dimensional spaces.

\begin{definition}[Persistence of excitation]\label{ass:pe}
Let $\FF$  be a Hilbert space and $\YY$ be a Banach space.
A continuous signal $g:\R_+\to \FF$ is persistently exciting (PE) with respect to a 
bounded linear operator $P\in\LL(\FF,\YY)$
if there exist positive constants $T$ and $\kappa$ such that
\begin{equation}\label{eq:pe}
    \int_t^{t+T}|\langle g(\tau), x\rangle_\FF|^2\dd\tau \geq \kappa \|Px\|_\YY^2,\quad \forall x\in\FF, \forall t\geq0.
\end{equation}
\end{definition}

\begin{remark}\label{rem:finite}
If $\FF=\YY$ is finite-dimensional and $P$ is a self-adjoint positive-definite operator, then Definition~\eqref{ass:pe} coincides with the usual notion of persistence of excitation since all norms on $\FF$ are equivalent.
However, if $\FF = \YY$ is infinite-dimensional, then there does not exist any PE signal with respect to the identity operator on $\FF$. (Actually, it is a characterization of the infinite dimensionality of $\FF$). Indeed, if $P=\Id_\FF$, then \eqref{eq:pe} at $t=0$ together with the spectral theorem for compact operators implies that $\int_0^{T}g(\tau)g(\tau)^*\dd \tau$ is not a compact operator, which is in contradiction with the fact that the sequence of finite range operators $\sum_{j=0}^Ng(\frac{jT}{N})g(\frac{jT}{N})^*$ converges to it as $N$ goes to infinity.
This is the reason for which we introduce this new PE condition which is feasible even if $\FF$ is infinite-dimensional.
Indeed, $P$ induces a semi-norm on $\FF$ that is weaker than or equivalent to $\|\cdot\|_\FF$.
\end{remark}

\begin{remark}
When $\FF$ is infinite-dimensional, note that there exist signals that are PE with respect to an operator $P$ inducing a norm on $\FF$ (weaker than $\|\cdot\|_\FF$), and not only a semi-norm. For example, consider $\FF = \YY =l^2(\N, \R)$ the Hilbert space of square summable real sequences. The signal $g:\R_+\to \FF$ defined by $g(\tau) = (\frac{\sin(k\tau)}{k^2})_{k\in\N}$ is
PE
with respect to $P:\FF\to\FF$ defined by $P(x_k)_{k\in\N} = \Big(\frac{x_k}{k^2}\Big)_{k\in\N}$ with constants $T=2\pi$ and $\kappa = \pi$ since $\int_0^{2\pi}\sin^2(k\tau)\dd\tau = \pi$ for all $k\in\N$.
\end{remark}

\begin{remark}\label{rem:pe}
If $\FF = L^2(\Omega, \R^n) $ for some positive integer $n$ and if $P$ is a HS integral operator with kernel $\rho$, then $\forall x\in\FF, \forall t\geq0$, equation \eqref{eq:pe} is equivalent to
\begin{equation}\label{eq:peHS}
\int_t^{t+T}|\langle g(\tau), x\rangle_\FF|^2\dd\tau \geq \kappa \int_{\Omega} \Bigg| \int_{\Omega} \rho(r, r')v(r') \dd r'\Bigg|^2\dd r.
\end{equation}
As explained in Remark~\ref{rem:finite}, the role of $\rho$ is to weaken the norm with respect to which $g$ has to be PE. If one changes the Lebesgue measure for the counting measure over a finite set $\Omega$ as suggested in Remark~\ref{rem:mes}, a possible choice of $\rho$ is the Dirac mass: $\rho(r, r') = 1$ if $r=r'$, $0$ otherwise. In that case, $\FF$ is finite-dimensional, and we recover the usual PE notion.
\end{remark}

Now, let us state the main theorem of this section that solves Problem~\ref{pb:obs}.

\begin{theorem}[Observer convergence]\label{th:obs}
Suppose that Assumptions~\ref{ass:wp} and \ref{ass:diss} are satisfied.
Define
$\alpha^* := \frac{\ell_{12}^2 \|w_{12}\|^2_{(\FF_{w_{12}}, \|\cdot\|)}}{2(1-\ell_{22}^2\|w_{22}\|^2_{(\FF_{w_{22}}, \|\cdot\|)})}$.
Then, for all $\alpha>\alpha^*$,
for all $u_1, u_2 \in C^0(\R_+\times \Omega, \R^{n_i})$,
any solution of \eqref{eq:wcij}-\eqref{eq:eps} is such that
$$\lim_{t\to+\infty}\|\zeps_1(t)\|_{\FF_{z_1}}= \lim_{t\to+\infty} \|\zeps_2(t)\|_{\FF_{z_2}}=0$$
and
$\|\weps_{11}(t)\|_{(\FF_{w_{11}}, \|\cdot\|_F)}$
and
$\|\weps_{12}(t)\|_{(\FF_{w_{12}}, \|\cdot\|_F)}$
remain bounded for all $t\geq0$.
\smallskip

Moreover, for any solution of \eqref{eq:wcij}, the corresponding error system \eqref{eq:eps} is uniformly Lyapunov stable at the origin, that is, for all $\eps>0$, there exists $\delta>0$ such that, if $$\|\zeps_{1}(t_0), \zeps_{2t_0}, \weps_{11}(t_0), \weps_{12}(t_0)\|_{\FF_{z_1}\times C^0([-\dbar, 0], \FF_{z_2})\times \FF_{w_{11}}\times \FF_{w_{12}}}\leq \delta$$ for some $t_0\geq0$,
then
$$\|\zeps_{1}(t), \zeps_{2}(t), \weps_{11}(t), \weps_{12}(t)\|_{\FF_{z_1}\times\FF_{z_2}\times \FF_{w_{11}}\times \FF_{w_{12}}}\leq \eps$$ for all $t\geq t_0$.
\smallskip

Furthermore,
if $S_{11}$ and $S_{12}$ are differentiable,
$u_1$ and $u_2$ are bounded\footnote{This assumption is missing in \cite{brivadis:hal-03660185} while it is implicitly used in the proof.} and
if $\rho_1 \in L^2(\Omega^2, \R^{n_1\times n_1})$ and $\rho_2\in L^2(\Omega^2, \R^{n_2\times n_2})$ are self-adjoint positive-definite 
kernels such that the signal $g:t\mapsto((r, r')\mapsto (S_{11}(z_1(t-d_{11}(r, r'), r')), S_{12}(z_2(t-d_{12}(r, r'), r'))))$ is PE  with respect to 
$P\in\LL(L^2(\Omega^2, \R^{n_1+n_2}), L^2(\Omega, \R^{n_1+n_2}))$ defined by
\begin{align*}
    &(P(x_1, x_2))(r) := \Bigg(\int_{\Omega^2} \rho_1(r, r') x_1(r'', r')\dd r''\dd r', \\
    &\int_{\Omega^2} \rho_2(r, r') x_2(r'', r')\dd r''\dd r'\Bigg)
\end{align*}
for all $(x_1, x_2)\in L^2(\Omega^2, \R^{n_1+n_2})$ and all $r\in\Omega$,
then
\begin{align}
   & \lim_{t\to+\infty} \|\weps_{11}(t)\circ \rho_1\|_{(\FF_{w_{11}}, \|\cdot\|_F)}\nonumber \\
    &= \lim_{t\to+\infty} \|\weps_{12}(t)\circ \rho_2\|_{(\FF_{w_{12}}, \|\cdot\|_F)}=0.\label{eq:conv_w}
\end{align}
\end{theorem}

The proof of Theorem~\ref{th:obs} is postponed to Section~\ref{sec:proof_obs}.

\begin{remark}\label{rem:hg}
    In the case where all the state is measured, i.e., $n_2 = 0$, note that $\alpha^* = 0$. Hence, under the PE assumption on $g$, the convergence of $\what_{11}\circ\rho_1$ towards $w_{11}$ is guaranteed for any positive observer gain $\alpha$. This means that the observer does not rely on any high-gain approach. This fact will be of importance in Section~\ref{sec:stab}, to show that the controller answering Problem~\ref{pb:stab} is not high-gain when the full state is measured, contrary to the approach developed in \cite{DECH16}.
\end{remark}

\begin{remark}\label{rem:blur}
    The obtained estimations of the kernels $w_{11}$ and $w_{12}$ in $(\FF_{w_{11}}, \|\cdot\|_F)$ is blurred by the kernels $\rho_1$ and $\rho_2$. The stronger is the semi-norm induced by $\rho_j$ (which is a norm if and only if $\rho_j$ is positive-definite), the stronger is the PE assumption, and the finer is the estimation of $w_{1j}$. In particular, if the counting measure replaces the Lebesgue measure over a finite set $\Omega$ and $\rho_j$ is a Dirac mass as suggested in Remark~\ref{rem:pe}, then $\weps_{1j}\circ \rho_j = \weps_{1j}$, hence the convergence of $\what_{1j}$ to $w_{1j}$ obtained in Theorem~\ref{th:obs} is in the topology of $L^2(\Omega^2, (\R^{n_i\times n_j}, \|\cdot\|_F))$, i.e., coefficientwise.
\end{remark}

\begin{remark}
The main requirement of Theorem~\ref{th:obs} lies in the persistence of excitation requirement, which is a common hypothesis to ensure convergence of adaptive observers (see, for instance, \cite{BESANCON2000271, FARZA20092292, sastry1990adaptive} in the finite-dimensional context and \cite{DEMETRIOU19965346, 761927} in the infinite-dimensional case).
Roughly speaking, it states that the parameters to be estimated are sufficiently ``excited'' by the system dynamics. However, this assumption is difficult to check in practice since it depends on the trajectories of the system itself.
In Section~\ref{sec:num}, we choose in numerical simulations a persistently exciting input $(u_1, u_2)$ in order to generate persistence of excitation in the signal $(S_{11}(z_1-d_{11}), S_{12}(z_2-d_{12}))$. This strategy seems to be numerically efficient, but the theoretical analysis of the link between the persistence of excitation of $(u_1, u_2)$ and that of $(S_{11}(z_1-d_{11}), S_{12}(z_2-d_{12}))$ remains an open question, not only in the present work but also for general classes of adaptive observers. This issue is further investigated in Section~\ref{sec:sim}, where we look for a feedback law allowing simultaneous kernel estimation and practical stabilization.
Another approach could be to design an observer not relying on PE, inspired by \cite{8786148,https://doi.org/10.48550/arxiv.2112.05497} for example. These methods, however, do not readily extend to the infinite-dimensional delayed context that is considered in the present paper. They could be investigated in future works. Finally, we emphasize that the persistence of excitation assumption is not required for the convergence of $\hat z_2$ (it is only used to make $\hat w_{1j}$ converge, meaning to estimate the synaptic distribution of neurons projecting to Population 1).
\end{remark}

\begin{remark}
According to Definition~\ref{ass:pe}, the PE assumption on $g$ in Theorem~\ref{th:obs} can be rewritten as follows: there exist positive constants $T$ and $\kappa$ such that, for all $(x_1, x_2)\in L^2(\Omega^2, \R^{n_1+n_2})$,
\begin{align}\label{eq:perem}
    &\int_0^T \Bigg|\sum_{j=1}^2 \int_{\Omega^2} g_j(t+\tau, r, r')^Tx_j(r, r') \dd r \dd r'\Bigg|^2 \dd \tau \nonumber 
 \\
&\geq\kappa \sum_{j=1}^2 \int_{\Omega} \Bigg| \int_{\Omega^2} \rho_j(r, r') x_j(r'', r')\dd r'' \dd r'\Bigg|^2\dd r
\end{align}
This characterization will be used in the proof of Theorem~\ref{th:obs}.
\end{remark}

\begin{remark}
The choice of the operator $P$ is a crucial part of Theorem~\ref{th:obs}.
First, its null space is given by $\ker P = \{(x_1, x_2)\in L^2(\Omega^2, \R^{n_1+n_2}) \mid \int_\Omega x_j(r, \cdot)\dd r \in \ker P_j, \forall j\in\{1,2\}\} = \{(x_1, x_2)\in L^2(\Omega^2, \R^{n_1+n_2}) \mid \|\int_\Omega x_j(r, \cdot)\dd r\|_{L^2(\Omega, \R^{n_j})} =0, \forall j\in\{1,2\}\}$ (where $P_j$ denotes the HS integral operator of kernel $\rho_j$), since $\rho_1$ and $\rho_2$ are positive-definite.
Secondly, remark that $P$ can be written as a block-diagonal operator, with two blocks in $\LL(L^2(\Omega, \R^{n_1}))$ and $\LL(L^2(\Omega^2, \R^{n_2}))$, respectively.
Roughly speaking, this means that the PE signal $g$ must excite ``separately'' on its two components so that we are able to distinguish them and to reconstruct separately $w_{11}$ and $w_{12}$.
Finally, note that if $d_{1j}$ does not depend on $r$ (i.e., $d_{1j}(\cdot, r')$ is constant for all $r'\in\Omega$), then neither does $g_j$ (we write $g_j(t, r') := g_j(t, r, r')$ by abuse of notations), and the kernel of $P$ simply means that we do not require to excite the system along $r$. In other words, in that case, \eqref{eq:perem} can be rewritten as
\begin{align*}
    &\int_0^T \Bigg|\sum_{j=1}^2 \int_{\Omega} g_j(t+\tau, r')^T X_j(r') \dd r'\Bigg|^2 \dd \tau \nonumber \\
\geq&
\kappa \sum_{j=1}^2 \int_{\Omega} \Bigg| \int_{\Omega} \rho_j(r, r') X_j(r') \dd r'\Bigg|^2\dd r
\end{align*}
where $X_j$ defined by $X_j(r') = \int_\Omega x_j(r, r')\dd r$ spans $L^2(\Omega, \R^{n_j})$ as $x_j$ spans $L^2(\Omega^2, \R^{n_j})$, which means that $g$ is PE with respect to the HS integral operator having kernel $\diag(\rho_1, \rho_2)$, which is a self-adjoint positive-definite endomorphism of $L^2(\Omega, \R^{n_1+n_2})$.
\end{remark}

\begin{remark}
    One of the drawbacks of the observer \eqref{eq:obs} is that Theorem~\ref{th:obs} does not guarantee input-to-state stability (ISS; see, e.g., \cite{Sontag2008, mironchenko2023input}) of the error system with respect to perturbations of the measured output $z_1$ or model errors.
    This is an important issue in the context of neurosciences since model parameters are often uncertain.
    In particular, the assumption that $S_{ij}$ and $d_{ij}$ are known is based on models that may vary with time and with individuals.
    From a mathematical point of view, it is due to the fact that the Lyapunov function $V$ (see \eqref{eq:lyap}) used to investigate the system's stability cannot easily be shaped into a control Lyapunov function. Numerical experiments are performed in Section~\ref{sec:num} to investigate robustness to measurement noise.
    From a theoretical viewpoint, in order to obtain additional robustness properties, new observers should be investigated in order to obtain global exponential contraction of the error system, for example, inspired by \cite{https://doi.org/10.48550/arxiv.2111.02176}.
    In any case, the Lyapunov analysis performed in Section~\ref{step1} is still a bottleneck for proving the convergence of observers of this kind.
\end{remark}

\begin{remark}
    Note that, in Theorem 3.2 as in all the results of this paper, the maximal delay $\dbar$ can be arbitrarily large.
    The only assumption made on delays is that $d_{ij}$ are known.
\end{remark}

\subsection{Proof of Theorem~\ref{th:obs} (observer convergence)}
\label{sec:proof_obs}


\subsubsection{Step 1: \textnormal{Proof that $\lim_{t\to+\infty}\|\zeps_1(t)\|_{\FF_{z_1}}= \lim_{t\to+\infty} \|\zeps_2(t)\|_{\FF_{z_2}}=0$}}\label{step1}


In order to obtain the first part of the result, we seek a Lyapunov functional $V$ for the estimation error dynamics.
Inspired by the analysis performed in \cite{DECH16}, let us consider the following candidate Lyapunov function:
\begin{align}
    V(\zeps_{1}, \zeps_{2t}, \weps_{11}, \weps_{12}) :=& V^z_1(\zeps_{1}) + V^z_2(\zeps_{2}) + V^w_1(\weps_{11})  \nonumber \\
    &+V^w_2(\weps_{12})+ W_{1}(\zeps_{2t})+W_{2}(\zeps_{2t})\label{eq:lyap}
\end{align}
where, for all 
$(\zeps_{1}, \zeps_{2t}, \weps_{11}, \weps_{12})\in \FF_{z_1}\times C^0([-\dbar, 0], \FF_{z_2})\times\FF_{w_{11}}\times\FF_{w_{12}}$,
\begin{align}
    V_i^z(\zeps_i) &:= \frac{1}{2}\int_\Omega  \zeps_i(r)^\top\tau_i(r)\zeps_i(r)\dd r,\label{eq:Viz}\\
    V^w_j(\weps_{1j})
    &:= \frac{1}{2} \int_{\Omega^2} \Tr(\weps_{1j}(r, r')^\top\tau_j(r)\weps_{1j}(r, r'))\dd r'\dd r, \label{eq:Vjw}\\
    W_{i}(\zeps_{2t}) &:=
        \int_{\Omega^2}\gamma_i(r) \int_{-d_{i2}(r, r')}^0 |\zeps_{2t}(s, r')|^2\dd s\dd r'\dd r, \label{eq:Wi}
\end{align}
and $\gamma_i\in \LL^2(\Omega, \R)$, for $i\in\{1, 2\}$, are to be chosen later.

Computing the time derivative of these functions along solutions of \eqref{eq:eps}, we get:
\begin{align*}
    &\frac{\dd}{\dd t}V^z_1(\zeps_{1}(t))
    = -\alpha\int_\Omega |\zeps_1(t, r)|^2\dd r
    \\
    &\quad
    +\int_{\Omega^2} \zeps_1(t, r)^\top\Big(
    \weps_{11}(t, r, r')(S_{11}(z_1(t-d_{11}(r, r'), r')))
    \\
    &\quad
    +\weps_{12}(t, r, r')(S_{12}(\zhat_2(t-d_{12}(r, r'), r')))
    \\
    &\quad
    +w_{12}(t, r, r')(S_{12}(\zhat_2(t-d_{12}(r, r'), r'))\nonumber \\
    &-S_{12}(z_2(t-d_{12}(r, r'), r')))
    \Big)\dd r'\dd r,
\end{align*}
\begin{align*}
    &\frac{\dd}{\dd t}V^z_2(\zeps_{2}(t))
    = -\int_\Omega |\zeps_2(t, r)|^2\dd r
    \\
    &\quad
    +\int_{\Omega^2} \zeps_2(t, r)^\top\Big(
    w_{22}(t, r, r')(S_{22}(\zhat_2(t-d_{22}(r, r'), r'))
    \\
    &\quad-S_{22}(z_2(t-d_{22}(r, r'), r')))
    \Big)\dd r'\dd r,
\end{align*}
\begin{align*}
   & \frac{\dd}{\dd t}(V^w_1(\weps_{11}(t)) +V^w_2(\weps_{12}(t)))
    \\
    &= - \int_{\Omega^2}\Tr(\weps_{11}(t, r, r')^\top \zeps_1(t, r)S_{11}(z_1(t-d_{11}(r, r'), r'))^\top)\dd r'\dd r
    \\
    &\quad-\int_{\Omega^2}\Tr(\weps_{12}(t, r, r')^\top \zeps_1(t, r)S_{12}(\zhat_2(t-d_{12}(r, r'), r'))^\top)\dd r'\dd r
    \\
    &=-\int_{\Omega^2}\zeps_1(t, r)^\top\Big(\weps_{11}(r, r')S_{11}(z_1(t-d_{11}(r, r'), r'))
    \\
    &\qquad+\weps_{12}(r, r')S_{12}(\zhat_2(t-d_{12}(r, r'), r'))\Big)\dd r'\dd r,
\end{align*}
and
\begin{equation*}
\begin{aligned}
    \frac{\dd}{\dd t}W_{i}(\zeps_{2t}) = 
        \int_{\Omega^2} \gamma_i(r)(|\zeps_2(t, r')|^2 - |\zeps_2(t- d_{i2}(r, r'), r')|^2)\dd r'\dd r.
\end{aligned}
\end{equation*}
Combining the previous computations, we obtain that
\begin{align*}
    &\frac{\dd}{\dd t}V(\zeps_{1}(t), \zeps_{2t}, \weps_{11}(t), \weps_{12}(t))
    = -\alpha\|\zeps_1(t)\|^2_{\FF_{z_1}}
    \\
    &-\Big(1-\sum_{i=1}^2\int_\Omega \gamma_i(r)\dd r\Big)\|\zeps_2(t)\|^2_{\FF_{z_2}}+\sum_{i=1}^2N_i(\zeps_i(t), \zhat_{2t}, z_{2t})
    \\
    &\quad - \sum_{i=1}^2\int_{\Omega^2}\gamma_i(r) |\zeps_2(t- d_{i2}(r, r'), r')|^2\dd r'\dd r
\end{align*}
where
\begin{align*}
N_i(\zeps_i(t), \zhat_{2t}, z_{2t}) := &\int_{\Omega^2} \zeps_i(t, r)^\top\Big(
    w_{i2}(r, r')(S_{i2}(\zhat_2(t-d_{i2}(r, r'),\\
    & r'))-S_{i2}(z_2(t-d_{i2}(r, r'), r')))
    \Big)\dd r'\dd r.
\end{align*}

Let us provide a bound of $N_i$ by applying Cauchy-Schwartz and Young's inequalities.
\begin{align*}
    &|N_i(\zeps_i(t), \zhat_{2t}, z_{2t})|\\
    &=\Bigg|\int_{\Omega^2} \zeps_i(t, r)^\top
    w_{i2}(r, r')\Big(S_{i2}(\zhat_2(t-d_{i2}(r, r'), r'))\\
    &\hspace{2cm}-S_{i2}(z_2(t-d_{i2}(r, r'), r'))\Big)
    \dd r'\dd r\Bigg|
    \\
    &
    \leq \int_\Omega \Big|\zeps_i(t, r)\Big|
    \Bigg|\int_\Omega
    w_{i2}(r, r')
    \Big(S_{i2}(\zhat_2(t-d_{i2}(r,r'), r')
    \\&\hspace{2cm}-S_{i2}(z_2(t-d_{i2}(r,r'), r'))\Big)
    \dd r'\Bigg|
    \dd r
    \\
    &
    \leq \sqrt{\int_\Omega \Big|\zeps_i(t, r)\Big|^2\dd r}
    \Big(\int_\Omega \Bigg|\int_\Omega
    w_{i2}(r, r')
    \Big(S_{i2}(\zhat_2(t-d_{i2}(r,r'), r')\\
    &\hspace{2cm}
    -S_{i2}(z_2(t-d_{i2}(r,r'), r'))\Big)
    \dd r'\Bigg|^2 \dd r\Big)^{\frac{1}{2}}
    \\
    &
    \leq \frac{1}{2\eps_i}\|\zeps_i(t)\|^2_{\FF_{z_i}}
    + \frac{\eps_i}{2} \int_\Omega \Bigg|\int_\Omega
    \|w_{i2}(r, r')\|
    \Big|S_{i2}(\zhat_2(t-\\
    &d_{i2}(r,r'), r')
    -S_{i2}(z_2(t-d_{i2}(r,r'), r'))\Big|
    \dd r'\Bigg|^2 \dd r
    \\
    &
    \leq \frac{1}{2\eps_i}\|\zeps_i(t)\|^2_{\FF_{z_i}}
    + \frac{\eps_i}{2} \int_\Omega \int_\Omega
    \|w_{i2}(r, r')\|^2\dd r'
    \int_\Omega
    \Big|S_{i2}(\zhat_2\\
    &(t-d_{i2}(r,r'), r')
    -S_{i2}(z_2(t-d_{i2}(r,r'), r'))\Big|^2\dd r'
    \dd r
    \\
    &
    \leq \frac{1}{2\eps_i}\|\zeps_i(t)\|^2_{\FF_{z_i}}
    + \frac{\eps_i\ell_{i2}^2}{2} \int_{\Omega} \int_{\Omega} \|w_{i2}(r, r')\|^2\dd r' \int_{\Omega}|\zeps_2(t-\\
    &d_{i2}(r, r'), r')|^2\dd r'\dd r
\end{align*}
for all $\eps_{i}>0$.
Now, set $\gamma_i(r) := \frac{\eps_i\ell_{i2}^2}{2} \int_{\Omega} \|w_{i2}(r, r')\|^2\dd r'$ for all $i\in\{1,2\}$.
We finally get that
\begin{equation}
\begin{aligned}
    &\frac{\dd}{\dd t}V(\zeps_{1}(t), \zeps_{2t}, \weps_{11}(t), \weps_{12}(t))
    \leq -\left(\alpha- \frac{1}{2\eps_1}\right)\|\zeps_1(t)\|^2_{\FF_{z_1}}
    \\
    &\quad-\left(1- \sum_{i=1}^2\int_\Omega \gamma_i(r)\dd r - \frac{1}{2\eps_2}\right)\|\zeps_2(t)\|^2_{\FF_{z_2}}.
\end{aligned}
\end{equation}
In order to make $V$ a Lyapunov function,
it remains to choose the constants $\eps_i$.
By definition of $\gamma_i$, we have $\sum_{i=1}^2\int_\Omega \gamma_i(r)\dd r = \frac{1}{2}\sum_{i=1}^2 \eps_i\ell_{i2}^2\|w_{i2}\|_{(\FF_{w_{i2}}, \|\cdot\|)}^2$.
Recall that, by Assumption~\ref{ass:diss} and by definition of $\alpha^*$,
$\ell_{22}^2\|w_{22}\|_{(\FF_{w_{22}}, \|\cdot\|)}^2<1$ and 
$\alpha>\frac{\ell_{12}^2\|w_{12}\|_{(\FF_{w_{12}}, \|\cdot\|)}^2}{2(1 - \ell_{22}^2\|w_{22}\|_{(\FF_{w_{22}}, \|\cdot\|)}^2)}$. Pick
\begin{align*}
    &\eps_1 := \frac{1 - \ell_{22}^2\|w_{22}\|_{(\FF_{w_{22}}, \|\cdot\|)}^2}{\ell_{12}^2\|w_{12}\|_{(\FF_{w_{12}}, \|\cdot\|)}^2}\\
    \quad \text{and}\quad
    &\eps_2 := \frac{1 + \ell_{22}^2\|w_{22}\|_{(\FF_{w_{22}}, \|\cdot\|)}^2}{2\ell_{22}^2\|w_{22}\|_{(\FF_{w_{22}}, \|\cdot\|)}^2}.
\end{align*}
Then, there exist two positive constants $c_1$ and $c_2$,
given by $c_1 = \alpha-\alpha^*$ and $c_2 = \frac{1}{4}(1-\ell_{22}^2\|w_{22}\|_{(\FF_{w_{22}}, \|\cdot\|)}^2)$,
such that for any solution of \eqref{eq:eps},
\begin{equation}\label{eq:dV}
\begin{aligned}
    \frac{\dd}{\dd t}V(\zeps_{1}(t), \zeps_{2t}, \weps_{11}(t), \weps_{12}(t))
    &\leq -c_1\|\zeps_1(t)\|^2_{\FF_{z_1}}\\
    &\quad -c_2\|\zeps_2(t)\|^2_{\FF_{z_2}}.
\end{aligned}
\end{equation}
Since $\tau_i\in C^0(\Omega, \D^{n_i}_{++})$, $V_i^z$ and $V^w_j$ define norms that are equivalent to the norms of $\FF_{z_i}$ and $(\FF_{w_{1j}}, \|\cdot\|_F)$, respectively.
Hence, the error system is uniformly Lyapunov stable, since $V$ is non-increasing.
Moreover, 
$\zeps_i$,
and
$\weps_{1j}$
are bounded for $i, j\in\{1,2\}$.
Moreover, we have for all $t\geq0$,
\begin{align*}
    &\frac{\dd}{\dd t} V_1^z(\zeps_1(t)) 
    \leq -\alpha \|\zeps_1(t)\|^2_{\FF_{z_1}}
    \\
    &\quad + \int_{\Omega} |\zeps_1(t, r)|\Big| \int_{\Omega}
    \weps_{11}(t, r, r')(S_{11}(z_1(t-d_{11}(r, r'), r')))
    \\
    &\quad
    +\weps_{12}(t, r, r')(S_{12}(\zhat_2(t-d_{12}(r, r'), r')))
    \\
    &\quad
    +w_{12}(t, r, r')(S_{12}(\zhat_2(t-d_{12}(r, r'), r'))\\
    &\quad-S_{12}(z_2(t-d_{12}(r, r'), r')))
    \dd r'\Big|\dd r,
    \\
    &\leq -\alpha \|\zeps_1(t)\|^2_{\FF_{z_1}} + \frac{1}{2}\|\zeps_1(t)\|^2_{\FF_{z_1}}
    +  \frac{1}{2}\bar S_{11}^2\|\weps_{11}\|_{(\FF_{w_{11}}, \|\cdot\|)}^2
    \\
    &\quad+  \frac{1}{2}\bar S_{12}^2\|\weps_{12}\|_{(\FF_{w_{12}}, \|\cdot\|)}^2
    +  2\bar S_{12}^2\|w_{12}\|_{(\FF_{w_{12}}, \|\cdot\|)}^2
\end{align*}
and
\begin{align*}
    &\frac{\dd}{\dd t} V_2^z(\zeps_2) \leq - \|\zeps_2(t)\|^2_{\FF_{z_2}}
    + \int_{\Omega} |\zeps_2(t, r)| \Big| \int_{\Omega}
    w_{22}(t, r, r')
    \\
    &(S_{22}(\zhat_2(t-d_{22}(r, r'), r'))-S_{22}(z_2(t-d_{22}(r, r'), r'))) \dd r'\Big | \dd r,
    \\
    &\leq  - \|\zeps_2(t)\|^2_{\FF_{z_2}} + \frac{1}{2} \|\zeps_2(t)\|^2_{\FF_{z_2}}
    + 2\bar S_{22}^2\|w_{22}\|_{(\FF_{w_{22}}, \|\cdot\|)}^2.
\end{align*}
Hence $\frac{\dd}{\dd t} V_1^z(\zeps_1)$ and $\frac{\dd}{\dd t} V_2^z(\zeps_2)$ are bounded since
$\zeps_i$,
and
$\weps_{1j}$
are bounded for all $i, j\in\{1,2\}$.
Thus, according to Barbalat's lemma applied to $V_i^z(\zeps_i)$, $V_i^z(\zeps_i(t))\to0$ as $t\to+\infty$, hence
$\|\zeps_i(t)\|_{\FF_{z_i}}\to0$.

\begin{remark}\label{rem:diss}
    At this stage, note that $\frac{\dd}{\dd t}(V_2^z(\zeps_2) + W_2(\zeps_{2t}))\leq -c_2\|\zeps_2\|^2_{\FF_{z_2}}$ whenever $\gamma_2$ and $\eps_2$ are chosen as above. Hence, Assumption~\ref{ass:diss} implies that $\zeps_2$ is converging towards 0. This justifies that Assumption~\ref{ass:diss} is indeed a dissipativity assumption of the $z_2$ sub-system, i.e. a detectability assumption, since $\zhat_2$ has the same dynamics than $z_2$.
\end{remark}

\subsubsection{Step 2: \textnormal{Proof that $\lim_{t\to+\infty} \|\weps_{11}(t)P_X\|_{\FF_{w_{11}}}= \lim_{t\to+\infty} \|\weps_{12}(t)P_Y\|_{\FF_{w_{12}}}=0$}}


Now, assume that $t\mapsto(S_{11}(z_1(t-d_{11})), S_{12}(z_2(t-d_{12})))$ is PE with respect to $\rho$. The error dynamics \eqref{eq:eps} can be rewritten as
\begin{align*}
    \tau_1(r)\frac{\partial \zeps_1}{\partial t}(t, r) &= f_0(t, r) +\int_{\Omega}\weps_{11}(t, r, r')g_1(t, r, r')\dd r'\\
    &
    + \int_{\Omega}\weps_{12}(t, r, r')g_2(t, r, r')\dd r'
    \\
    \tau_1(r)\frac{\partial \weps_{11}}{\partial t}(t, r, r') &= f_1(t, r, r')
    \\
    \tau_2(r)\frac{\partial \weps_{12}}{\partial t}(t, r, r') &= f_2(t, r, r')
\end{align*}
where
$$
g_j(t, r, r') := S_{1j}(z_j(t-d_{1j}(r, r'), r')), \quad \forall j\in\{1,2\},
$$
\begin{align*}
&\|f_0(t)\|^2_{\FF_{z_1}}
:= \int_\Omega \Bigg|-\alpha \zeps_1(t, r) 
+ \int_\Omega \what_{12}(t, r, r')(S_{12}(\zhat_2(t-\\
&\qquad \qquad d_{12}(r, r'), r')) - S_{12}(z_2(t-d_{12}(r, r'), r')))\dd r'
\Bigg|^2\dd r
\\
&\leq 2\alpha^2\|\zeps_1(t)\|^2_{\FF_{z_1}}
+ 2\ell_{12}^2 \int_\Omega \int_\Omega \|\what_{12}(t, r, r')\|^2\dd r' \int_\Omega \|\zeps_{2}(t \\
&\qquad \qquad \qquad - d_{12}(r, r'), r')\|^2\dd r'\dd r
\\
&\leq 2\alpha^2\|\zeps_1(t)\|^2_{\FF_{z_1}}
+ 2\ell_{12}^2 \|\hat w_{12}(t)\|_{(\FF_{w_{12}}, \|\cdot\|)}^2
\sup_{s\in[-\dbar, 0]}\|\zeps_{2t}(s)\|^2_{\FF_{z_2}},
\end{align*}
\begin{align*}
\|f_1(t)\|_{(\FF_{w_{11}}, \|\cdot\|)}
&:= \int_{\Omega^2} |\zeps_1(t, r) S_{11}(z_1(t-d_{11}(r, r'), r'))|^2
\\
&\leq \bar S_{11}^2\mu(\Omega) \|\zeps_1(t)\|_{\FF_{z_1}}^2,
\end{align*}
and
\begin{align*}
\|f_2(t)\|_{(\FF_{w_{12}}, \|\cdot\|)}
&:= \int_{\Omega^2} |\zeps_1(t, r) S_{12}(\zhat_2(t-d_{12}(r, r'), r'))|^2
\\
&\leq \bar S_{12}^2\mu(\Omega) \|\zeps_1(t)\|_{\FF_{z_1}}^2.
\end{align*}
Recall that, according to \hyperlink{step1}{Step 1}, $\|\zeps_i(t)\|_{\FF_{z_i}}\to0$ for all $i\in\{1,2\}$,
and $\|\hat w_{12}(t)\|_{(\FF_{w_{12}}, \|\cdot\|)}\leq \|\weps_{12}(t)\|_{(\FF_{w_{12}}, \|\cdot\|)} + \|w_{12}(t)\|_{(\FF_{w_{12}}, \|\cdot\|)}$ remains bounded
as $t\to+\infty$.
Hence
$\|f_0(t)\|_{(\FF_{z_{1}}, \|\cdot\|)}\to0$ and
$\|f_j(t)\|_{(\FF_{w_{1j}}, \|\cdot\|)}\to0$
as $t\to+\infty$ for all $j\in\{1,2\}$.
Moreover, $t\mapsto (g_1(t),g_2(t))$ is PE with respect to $\rho$ by assumption.

Applying twice Duhamel's formula
(once on $\zeps_1$, then once on $\weps_{1j}$)
, we get that for all $t, \tau\geq0$,
\begin{align*}
&\tau_1(r)\zeps_1(t+\tau, r)
= \tau_1(r)\zeps_1(t, r) + \int_0^\tau f_0(t+s, r)\dd s\\
&+\sum_{j=1}^2\int_{\Omega}\int_0^\tau \weps_{1j}(t+s, r, r')g_j(t+s, r, r')\dd s\dd r'
\\
&= \tau_1(r)\zeps_1(t, r) + \int_0^\tau f_0(t+s, r)\dd s
+\sum_{j=1}^2\int_{\Omega} \weps_{1j}(t, r, r')\\
&\int_0^\tau g_j(t+s, r, r')\dd s\dd r'+\sum_{j=1}^2 \int_{\Omega}\int_0^\tau\int_0^s\tau_j(r)^{-1}f_j(t+\\
&\sigma, r, r')g_j(t+s, r, r')\dd \sigma\dd s\dd r'.
\end{align*}
For any $t, T\geq0$, define $\OO(t, T):=\int_0^T\int_\Omega \zeps_1(t+\tau, r)^\top \tau_1(r)^2 \zeps_1(t+\tau, r) \dd r\dd \tau$. Since $\|\zeps_1(t)\|_{\FF_{z_1}}\to0$, $\OO(t, T)\to0$ as $t\to+\infty$ for all $T\geq0$. Moreover,
\small
\begin{align*}
    &\OO(t, T) = \int_0^T \int_\Omega \Bigg|
    \tau_1(r)\zeps_1(t, r) + \int_0^\tau f_0(t+s, r)\dd s +\sum_{j=1}^2
    \\
    & \int_{\Omega}\int_0^\tau\int_0^s\tau_j(r)^{-1}f_j(t+\sigma, r, r')g_j(t+s, r, r')\dd \sigma\dd s\dd r'
    \Bigg|^2\dd r\dd\tau
    \\
    &+
    \int_0^T \int_\Omega \Bigg|
    \sum_{j=1}^2\int_{\Omega} \weps_{1j}(t, r, r')\int_0^\tau g_j(t+s, r, r')\dd s\dd r'
    \Bigg|^2\dd r\dd\tau
    \\
    &+2\int_0^T \int_\Omega
    \Big(\sum_{j=1}^2\int_{\Omega} \weps_{1j}(t, r, r')\int_0^\tau g_j(t+s, r, r')\dd s\dd r'\Big)^\top
    \\
    &\quad
    \Big(\tau_1(r)\zeps_1(t, r) + \int_0^\tau f_0(t+s, r)\dd s +\sum_{j=1}^2 \int_{\Omega}\int_0^\tau\int_0^s\tau_j(r)^{-1}
    \\&\quad
    f_j(t+\sigma, r, r')g_j(t+s, r, r')\dd \sigma\dd s\dd r'\Big)
    \dd r\dd\tau
\end{align*}
\normalsize
Since $\|\zeps_1(t)\|_{\FF_{z_1}}\to0$, $\|f_0(t)\|_{(\FF_{z_{1}}, \|\cdot\|)}\to0$ and
$\|f_j(t)\|_{(\FF_{w_{1j}}, \|\cdot\|)}\to0$ as $t\to+\infty$,
$|g_j(t, r, r')|\leq \bar S_{1j}$ for all $t\geq 0$ and all $r, r'\in \Omega$,
and $t\mapsto\|\weps_{1j}(t)\|_{\FF_{w_{1j}}}$ is bounded,
we get that for any $T>0$,

\begin{align}
    \lim_{t\to+\infty} \int_0^T \int_\Omega \Bigg|
   & \sum_{j=1}^2\int_{\Omega} \weps_{1j}(t, r, r')\nonumber \\
    &\int_0^\tau g_j(t+s, r, r')\dd s\dd r'
    \Bigg|^2\dd r\dd\tau = 0.\label{eq:conv1}
\end{align}
For all $t, \tau \geq 0$ and all $r\in\Omega$, define $h(t, \tau, r) := \sum_{j=1}^2\int_{\Omega} \weps_{1j}(t, r, r')\int_0^\tau g_j(t+s, r, r')\dd s\dd r'$.
By \eqref{eq:conv1}, $\|h(t, \cdot)\|_{L^2((0, T), \FF_{z_1})}\to0$ as $t\to+\infty$.
Note that $\frac{\partial h}{\partial \tau}(t, \tau, r) = \sum_{j=1}^2\int_{\Omega} \weps_{1j}(t, r, r')g_j(t+\tau, r, r')\dd r'$ hence $h(t, \cdot)\in W^{1, 2}((0, T), \FF_{z_1})$ and is bounded since $g_j$'s are bounded.
Moreover,
since $u_j$ is supposed to be bounded for all $j\in\{1,2\}$,
$\frac{\dd z_i}{\dd t}$ is also bounded according to Proposition~\ref{th:bibs}.
Hence, since $S_{1j}$'s are differentiable with bounded derivative, 
$\frac{\partial^2 h}{\partial \tau^2}(t, \tau, r) = \sum_{j=1}^2\int_{\Omega} \weps_{1j}(t, r, r')\frac{\partial g_j}{\partial \tau}(t+\tau, r, r')\dd r'$ is well-defined and bounded.
Therefore, for all $t\geq 0$, $h(t, \cdot)\in W^{2, 2}((0, T), \FF_{z_1})$
and $\|h(t, \cdot)\|_{W^{2, 2}((0, T), \FF_{z_1})}\leq c_3$ for some positive constant $c_3$ independent of $t$.
According to the interpolation inequality (see, e.g., \cite[Section II.2.1]{temam1986infinite}),
\begin{equation*}
    \|h(t, \cdot)\|_{W^{1, 2}((0, T), \FF_{z_1})}^2\leq c_3\|h(t, \cdot)\|_{L^2((0, T), \FF_{z_1})}.
\end{equation*}
Thus $\|\frac{\partial h}{\partial \tau}(t, \tau)\|_{L^2((0, T), \FF_{z_1})}\to0$, meaning that
\begin{equation}\label{eq:conv2}
    \lim_{t\to+\infty} \int_0^T \int_\Omega \Bigg|
    \sum_{j=1}^2\int_{\Omega} \weps_{1j}(t, r, r')g_j(t+\tau, r, r')\dd r'
    \Bigg|^2\dd r\dd\tau = 0.
\end{equation}
Let $(e_k)_{k\in\N}$ be a Hilbert basis of $\FF_{z_1}$. We have, for all $t\geq0$,
\small
\begin{align*}
    &\int_0^T \int_\Omega \Bigg|
    \sum_{j=1}^2\int_{\Omega} \weps_{1j}(t, r, r')g_j(t+\tau, r, r')\dd r'
    \Bigg|^2\dd r\dd\tau\\
    &= \int_0^T \sum_{k\in\N} \Bigg|\sum_{j=1}^2\int_{\Omega^2} \Big(\weps_{1j}(t, r, r')g_j(t+\tau, r, r')\Big)^\top e_k(r)\dd r'\dd r\Bigg|^2\dd \tau
    \\
    &= \sum_{k\in\N} \int_0^T \Bigg|\sum_{j=1}^2\int_{\Omega^2} g_j(t+\tau, r, r')^\top\weps_{1j}(t, r, r')^\top e_k(r)\dd r'\dd r\Bigg|^2\dd \tau.
\end{align*}
\normalsize
Now, since $g$ is PE with respect to $\rho$ over $L^2(\Omega^2, \R^{n_1+n_2})$, there exist positive constants $T$ and $\kappa$ such that \eqref{eq:perem} holds.
for all $x_j\in L^2(\Omega^2, \R^{n_j})$, $j\in\{1,2\}$, and all $t\geq0$.
Choosing $x_j(r, r') = \weps_{1j}(t, r, r')^\top e_k(r)$, we get
\begin{multline*}
    \int_0^T \int_\Omega \Bigg|
    \sum_{j=1}^2\int_{\Omega} \weps_{1j}(t, r, r')g_j(t+\tau, r, r')\dd r'
    \Bigg|^2\dd r\dd\tau
    \\
    \geq \sum_{k\in\N} \kappa \sum_{j=1}^2 \int_{\Omega} \Bigg| \int_{\Omega^2} \rho_j(r, r') \weps_{1j}(t, r'', r')^\top e_k(r'')\dd r'' \dd r'\Bigg|^2\dd r.
\end{multline*}
On the other hand,
\begin{align*}
&\sum_{j=1}^2  \|\weps_{1j}(t)\circ \rho_j\|^2_{(\FF_{w_{1j}}, \|\cdot\|_F)}
\\
&= \sum_{j=1}^2 \|\weps(t)\circ \rho_j\|^2_{L^2(\Omega^2, (\R^{n_1\times n_j}, \|\cdot\|_F))}
\\
&= \sum_{j=1}^2 \|\rho_j\circ \weps(t)^*\|^2_{L^2(\Omega^2, (\R^{n_j\times n_1}, \|\cdot\|_F))}
\\
&= \sum_{j=1}^2 \sum_{k\in\N} \int_{\Omega} \Bigg| \int_{\Omega} (\rho_j\circ \weps(t)^*)(r, r'') e_k(r'')\dd r'' \Bigg|^2\dd r
\\
&= \sum_{j=1}^2 \sum_{k\in\N} \int_{\Omega} \Bigg| \int_{\Omega^2} \rho_j(r, r') \weps_{1j}(t, r'', r')^\top \dd r'e_k(r'')\dd r'' \Bigg|^2\dd r
\\
&=   \sum_{j=1}^2 \sum_{k\in\N} \int_{\Omega} \Bigg| \int_{\Omega^2} \rho_j(r, r') \weps_{1j}(t, r'', r')^\top e_k(r'')\dd r'' \dd r'\Bigg|^2\dd r
\\
&\leq \frac{1}{\kappa} \int_0^T \int_\Omega \Bigg|
    \sum_{j=1}^2\int_{\Omega} \weps_{1j}(t, r, r')g_j(t+\tau, r, r')\dd r'
    \Bigg|^2\dd r\dd\tau.
\end{align*}
Thus, by \eqref{eq:conv2},
\begin{align*}
    0=&\lim_{t\to+\infty} \|\weps_{12}(t)\circ \rho_2\|_{(\FF_{w_{12}}, \|\cdot\|_F)}\\
    =& \lim_{t\to+\infty} \|\weps_{11}(t)\circ \rho_1\|_{(\FF_{w_{11}}, \|\cdot\|_F)},
\end{align*}
\normalsize
which concludes the proof of Theorem~\ref{th:obs}.


\section{Adaptive control}\label{sec:stab}

In order to tackle Problem~\ref{pb:stab}, we now introduce an adaptive controller based on the previous observer design.




\subsection{Exact stabilization}\label{sec:exact}

Let $\zref1\in \FF_{z_1}$
be a constant reference signal at which we aim to stabilize $z_1$.
The $z_2$ dynamics of \eqref{eq:wcij} can be written as
\begin{equation}
\label{eq:z2ref}
\begin{aligned}
\tau_2(r)\frac{\partial z_2}{\partial t}(t, r)
&= -z_2(t, r) + \int_{\Omega}w_{22}(r, r')S_{22}(z_2(t-d_{22}(r
\\
&\quad, r'), r'))\dd r'+ v_{2, \mathrm{ref}}(r) + v_2(t, r)
\end{aligned}
\end{equation}
where $$v_{2, \mathrm{ref}}(r) := \int_{\Omega}w_{21}(r, r')S_{21}(\zref1(r'))\dd r'$$
and \begin{align*}v_2(t, r) := u_2(t, r) + \int_{\Omega}&w_{21}(r, r')(S_{21}(z_1(t,r)) \\
&- S_{21}(\zref{1}(r')))\dd r'.\end{align*}

When $z_1$ is constantly equal to $\zref{1}$ and $u_2=0$, we have $v_2 = 0$. Hence, according to \cite[Proposition 3.6]{Faugeras:2008wx},
system~\eqref{eq:z2ref}
admits, in that case, a stationary solution, i.e., there exists $\zref2\in\FF_{z_2}$ such that
$$
\zref2(r) = \int_{\Omega}w_{22}(r, r')S_{22}(\zref2(r'))\dd r'+ v_{2, \mathrm{ref}}(r),\quad \forall r\in\Omega.
$$
Moreover, we have the following stability result.
\begin{lemma}[\hspace{-0.5pt}\protect{\cite[Proposition 1]{DECH16}}]\label{lem:iss}
Under Assumption~\ref{ass:diss}, system~\eqref{eq:z2ref} is input-to-state stable (\emph{ISS}) at $\zref{2}$ with respect to $v_2$,
that is, there exist functions $\beta$ of class $\KL$ and $\nu$ of class $\Kinf$
such that, for all $z_{2, 0}\in C^0([-\dbar, 0], \FF_{z_2})$ and all $v_2\in C^0(\R_+, \FF_{z_2})$, the corresponding solution $z_2$ of \eqref{eq:z2ref} satisfies, for all $t\geq 0$,
\begin{align*}
    \|z_2(t)-\zref{2}\|_{\FF_{z_2}} \leq &\beta(\|z_2(0)-\zref{2}\|_{\FF_{z_2}}, t) \\
    &+ \nu\Big(\sup_{\tau\in[0, t]}\|v_2(\tau)\|_{\FF_{z_2}}\Big).
\end{align*}
\end{lemma}
\begin{remark}
    The result given in \cite[Proposition 1]{DECH16} holds for systems whose dynamics is of the form (after a change of variable)
\begin{align*}
    \tau_2(r)\frac{\partial z_2}{\partial t}(t, r) =&\, -z_2(t, r) + S_{22}(\int_{\Omega}w_{22}(r, r')z_2(t-\\
    &d_{22}(r, r'), r')\dd r')+ v_{2, \mathrm{ref}}(r) + v_2(t, r),
\end{align*}
    which is slightly different from \eqref{eq:z2ref}.
    However, one can easily check that this modification does not impact at all the proof given in \cite{DECH16}. In particular, the control Lyapunov function given in \cite{DECH16} remains a Lyapunov function for \eqref{eq:z2ref} (where $v_2$ is the input).
\end{remark}

In particular, Assumption~\ref{ass:diss} implies that $\zref{2}$ is unique. In the case where $\zref1 = 0$ and $S_{2j}(0) = 0$ for all $j\in\{1,2\}$, we also have $\zref2 = 0$.

We aim to define a dynamic output feedback law that stabilizes $(z_1, z_2)$ at the reference $(\zref1, \zref2)$.
We propose the following feedback strategy:
for all $t\geq0$ and all $r\in\Omega$, set
\begin{equation}\label{eq:cont}
\begin{aligned}
&\begin{aligned}
        u_1(t, r) =&\, - \alpha (z_1(t, r) - \zref1(r)) + z_1(t, r)\\
        &\,- \int_{r'\in\Omega}\what_{11}(t, r, r')S_{11}(z_1(t-d_{11}(r, r'), r'))\dd r'
        \\
        &\,- \int_{r'\in\Omega}\what_{12}(t, r, r')S_{12}(\hat z_2(t-d_{12}(r, r'), r'))\dd r',
\end{aligned}
\\
&u_2(t, r) = 0,
\end{aligned}
\end{equation}
where $\alpha>0$ is a tunable controller gain.
The motivation of this controller is that, when $w_{1j} = \what_{1j}$ and $\zhat_2 = z_2$, the resulting dynamics of $z_1$ is $\tau_1\frac{\partial z_1}{\partial t} = - \alpha (z_1- \zref1)$. Hence, since $z_2$ has a contracting dynamics by Assumption~\ref{ass:diss}, this would lead $(z_1, z_2)$ towards $(\zref1, \zref2)$.
Note that for the controller \eqref{eq:cont}, the resulting dynamics of $\zhat_1$ in \eqref{eq:obs} is $\tau_1\frac{\partial \zhat_1}{\partial t} = -\alpha (\zhat_1- \zref1)$.
Hence, since the choice of the initial condition of the observer is free, one particular instance of the closed-loop observer is given by $\zhat_1(t, r) = \zref1(r)$ for all $t\geq0$ and all $r\in\Omega$.
For this reason, in the stabilization strategy, we can reduce the dimension of the observer by setting $\zhat_1 = \zref1$, i.e., $\zeps_1 = \zref1-z_1$. Finally, the closed-loop system that we investigate can be rewritten as system~\eqref{eq:wcij} coupled with the controller \eqref{eq:cont} and the observer system given by
\begin{equation}\label{eq:obs_closed}
\begin{aligned}
&\begin{aligned}
    &\tau_2(r)\frac{\partial \zhat_2}{\partial t}(t, r) = 
    +\int_{\Omega}w_{21}(r, r')S_{21}(z_1(t-d_{21}(r, r'), r'))\dd r'
    \\
    &
    +\int_{\Omega}w_{22}(r, r')S_{22}(\zhat_2(t-d_{22}(r, r'), r'))\dd r'-\zhat_2(t, r)
\end{aligned}
\\
&
    \tau_1(r)\frac{\partial \what_{11}}{\partial t}(t, r, r') = (z_1(t, r)-\zref1(r))\\
    &\hspace{3.5cm}\cdot S_{11}(z_1(t-d_{11}(r, r'), r'))^\top
\\
&
\tau_1(r)\frac{\partial \what_{12}}{\partial t}(t, r, r') = (z_1(t, r)-\zref1(r))\\
&\hspace{3.5cm}\cdot S_{12}(\zhat_2(t-d_{12}(r, r'), r'))^\top.
\end{aligned}
\end{equation}


First, let us ensure the well-posedness of the resulting closed-loop system.

\begin{proposition}[Closed-loop well-posedness]\label{th:wp_closed}
Suppose that Assumption~\ref{ass:wp} is satisfied.
For any initial condition $(z_{1,0}, z_{2,0}, \zhat_{2,0}, \what_{11,0}, \what_{12,0})\in C^0([-\dbar, 0], \FF_{z_1})\times C^0([-\dbar, 0], \FF_{z_2})^2\times\FF_{w_{11}}\times\FF_{w_{12}}$,
the closed-loop system \eqref{eq:wcij}-\eqref{eq:cont}-\eqref{eq:obs_closed} admits a unique corresponding solution $(z_1, z_2, \zhat_2, \what_{11}, \what_{12})\in C^1([0, +\infty), \FF_{z_1}\times\FF_{z_2}^2\times\FF_{w_{11}}\times\FF_{w_{12}})\cap C^0([-\dbar, +\infty), \FF_{z_1}\times\FF_{z_2}^2\times\FF_{w_{11}}\times\FF_{w_{12}})$.
\end{proposition}
\begin{proof}
We adapt the proof of Proposition~\ref{th:wp}.
The main difference is that, since the system is in closed loop, the state variables $(z_1, z_2)$ cannot be taken as external inputs in the observer dynamics.
Therefore, we consider
the map $F:\R_+ \times C^0([-\dbar, 0], \FF_{z_1})\times C^0([-\dbar, 0], \FF_{z_2})^2\times\FF_{w_{11}}\times\FF_{w_{12}}
\to
\FF_{z_1}\times  \FF_{z_2}^2\times\FF_{w_{11}}\times\FF_{w_{12}}
$
such that \eqref{eq:obs} can be rewritten as
$\frac{\dd}{\dd t}(z_1, z_2, \zhat_2,\what_{11},\what_{12})(t) = F(t, z_{1t}, z_{2t}, \zhat_{2t},\what_{11},\what_{12})$.
Since $\tau_i$ are continuous and positive, $S_{ij}$ are bounded and $w_{ij}$ are square-integrable over $\Omega^2$ and $d_i$ are continuous, the map $F$ is well-defined by the same arguments than \cite[Lemma 3.1.1]{FAYE2010561}.
Let us show that $F$ is continuous, and globally Lipschitz with respect to $(\zhat_{1}, \zhat_{1t},\what_{11},\what_{12})$, so that we can conclude with \cite[Lemma 2.1 and Theorem 2.3]{hale2013introduction}.
Define $F_1$ taking values in $\FF_{z_1}$, $F_2$ and $F_3$ taking values in $\FF_{z_2}$, $F_4$ taking values in $\FF_{w_{11}}$ and $F_5$ taking values in $\FF_{w_{12}}$ so that $F = (F_i)_{i\in\{1,2,3,4,5\}}$
From the proof of \cite[Lemma 3.1.1]{FAYE2010561}, $F_1$, $F_2$, and $F_3$ are continuous and globally Lipschitz with respect to the last variables.
From the boundedness and global Lipschitz continuity of $S$, $F_4$, and $F_5$ are also continuous and globally Lipschitz with respect to the last variables.
This concludes the proof of Proposition~\ref{th:wp_closed}.
\end{proof}

Now, the main theorem of this section can be stated.
\begin{theorem}[Exact stabilization]\label{th:stab}
Suppose that Assumptions~\ref{ass:wp} and \ref{ass:diss} are satisfied.
Define
$\alpha^* := \frac{\ell_{12}^2 \|w_{12}\|^2_{(\FF_{w_{12}}, \|\cdot\|)}}{2(1-\ell_{22}^2\|w_{22}\|^2_{(\FF_{w_{22}}, \|\cdot\|)})}$.
Then, for all $\alpha>\alpha^*$,
any solution of \eqref{eq:wcij}-\eqref{eq:cont}-\eqref{eq:obs_closed} is such that
\begin{align*}\lim_{t\to+\infty}\|z_1(t)-\zref1\|_{\FF_{z_1}}&=\lim_{t\to+\infty} \|z_2(t)-\zref2\|_{\FF_{z_2}}\\&=\lim_{t\to+\infty} \|\zhat_2(t)-z_2(t)\|_{\FF_{z_2}}=0\end{align*}
and
$\|\what_{11}(t)\|_{(\FF_{w_{11}}, \|\cdot\|_F)}$
and
$\|\what_{12}(t)\|_{(\FF_{w_{12}}, \|\cdot\|_F)}$
remain bounded for all $t\geq0$.
\smallskip

Moreover, the system \eqref{eq:wcij}-\eqref{eq:cont}-\eqref{eq:obs_closed} is uniformly Lyapunov stable at $(\zref1, \zref2, \zref2, w_{11}, w_{12})$,
that is, for all $\eps>0$, there exists $\delta>0$ such that, if \begin{align*}\|&z_{1}(t_0)-\zref1, z_{2t_0}-\zref2, \zhat_{2t_0}-\zref2, \what_{11}(t_0)-w_{11},\\
& \what_{12}(t_0)-w_{12}\|_{\FF_{z_1}\times C^0([-\dbar, 0], \FF_{z_2}^2)\times\FF_{w_{11}}\times\FF_{w_{12}}}\leq \delta\end{align*}
for some $t_0\geq0$,
then
\begin{align*}\|&z_{1}(t)-\zref1, z_2(t)-\zref2, \zhat_{2}(t)-\zref2, \what_{11}(t)-w_{11}, \\
&\what_{12}(t)-w_{12}\|_{\FF_{z_1}\times\FF_{z_2}^2\times\FF_{w_{11}}\times\FF_{w_{12}}}\leq \eps\end{align*}
for all $t\geq t_0$.
\smallskip


\end{theorem}
\begin{proof}
As in Section~\ref{sec:obs}, define the observer error $\zeps_2 = \zhat_2- z_2$ and $\weps_{1j} = \what_{1j}-w_{1j}$.
Define also $\zeps_1:=\zref{1}-z_1$.
Then $(\zeps_1, \zeps_2, \weps_{11}, \weps_{12})$ satisfies \eqref{eq:eps}.
Thus, the first part (i.e. the part without PE assumption) of Theorem~\ref{th:obs} can be applied.
It implies that $(\zeps_1, \zeps_2)\to0$ in $\FF_{z_1}\times \FF_{z_1}$, $(\weps_{11}, \weps_{12})$ remains bounded, and this autonomous system is uniformly Lyapunov stable at the origin.
Moreover, $z_2$ satisfies the dynamics \eqref{eq:z2ref}
with $v_2(t, r) = \int_{\Omega}w_{21}(r, r')(S_{21}(\zref{1}(r')-\zeps_1(t, r')) - S_{21}(\zref{1}(r')))\dd r'
$. In other words, $(\zeps_1, \zeps_2, \weps_{11}, \weps_{12})$ coupled with $z_2$ is a cascade system where $z_2$ is driven by the other variables. According to Lemma~\ref{lem:iss}, \eqref{eq:z2ref}
is ISS with respect to $v_2$. Thus, $z_2\to\zref{2}$ in $\FF_{z_2}$ and is uniformly Lyapunov stable at $\zref{2}$.
\end{proof}

Note that this exact stabilization result does not rely on any PE assumption, i.e. that the convergence of $\hat w_{1j}$ towards $w_{1j}$ is not crucial for stabilization. This can also be seen in the proof, since only the first part of Theorem \ref{th:obs} is used.

\begin{remark}
    Let us consider the case where the full state is measured, i.e., $n_2=0$. Then $\alpha^* = 0$, which means that the controller does not rely on any high-gain approach (see Remark~\ref{rem:hg}). This is the main difference between the present work and the approach developed in \cite{DECH16}, where the gain has to be chosen sufficiently large even in the fully measured context.
\end{remark}

\begin{remark}\label{rem:noconv}
    Convergence of the kernel estimation $\what_{1j}$ towards $w_{1j}$ is not investigated in Theorem~\ref{th:stab}. 
    Actually, one can apply the last part of Theorem~\ref{th:obs} to show that, under a PE assumption on the signal $g = (S_{11}(z_1(t-d_{11})), S_{12}(z_2(t-d_{12})))$, the kernel estimation is guaranteed in the sense of \eqref{eq:conv_w}. However, since the feedback law is made to stabilize the system, $z_1$ and $z_2$ are converging towards $(\zref1, \zref2)$.
    Hence, for $g$ to be PE with respect to $P$, one must have that for some positive constant $T$ and $\kappa$,
    \begin{equation*}
    T| (S_{11}(\zref1), S_{12}(\zref2))^* v|^2 \geq \kappa \|Px\|_\YY^2,\quad \forall x\in\FF, \forall t\geq0.
    \end{equation*}
    i.e. that $(S_{11}(\zref1), S_{12}(\zref2))^*$ induces a semi-norm on $\FF$ stronger than the one induced by $P$.
    This operator being a linear form, it implies that the kernel of $P$ must contain a hyperplane, which makes the convergence of $\what_{1j}$ towards $w_{1j}$ too much blurred by $P$ to claim that any interesting information on $w_{1j}$ can be reconstructed by this method, except when $\dim \FF$ = 1.
    In particular, if $(S_{11}(\zref1), S_{12}(\zref2)) = (0, 0)$, then $P=0$, hence no convergence of $\what_{1j}$ towards $w_{1j}$ is guaranteed.
    The objective of simultaneously stabilizing the neuronal activity while estimating the kernels is investigated in the next section.
\end{remark}





\subsection{Simultaneous kernel estimation and practical stabilization}\label{sec:sim}

The goal of this section is to propose an observer and a controller that allow to \emph{simultaneously} answer Problems~\ref{pb:obs} and \ref{pb:stab}. In order to do so, we make the following set of restrictions (in this subsection only)
for reasons that will be pointed in Remark~\ref{rem:rest}:
\begin{enumerate}[label = \textit{(\roman*)}]
    \item Exact stabilization is now replaced by \emph{practical} stabilization, that is, for any arbitrary small neighborhood of the reference, a controller that stabilizes the system within this neighborhood has to be designed.
    \label{rel:1}
    \item All the neuronal activity is measured, i.e., $n_2 = 0$. Hence, we set $z=z_1$, $w_{11} = w$, $u = u_1$, $S = S_{11}$, $\tau = \tau_{1}$ and $d = d_{11}$ to ease the notations.
    \label{rel:2}
    \item The state space $\FF_z$ is finite-dimensional. To do so, as suggested in Remark \ref{rem:mes}, we replace the Lebesgue measure with the counting measure and take $\Omega$ as a finite collection of $N$ points in $\R^q$, so that $\FF_z \simeq \R^{(N\times n)}$.
    \label{rel:3}
    \item The delay $d$ is constant. We write $\dbar := d(r, r')$ for all $r, r'\in\Omega$ by abuse of notations.
    \label{rel:5}
    \item The decay rate $\tau(r)$ is constant and all its components are equal, that is,
    $\tau(r) = \tau\Id_{\R^n}$ for all $r\in\Omega$ for some positive real constant $\tau$.
    \label{rel:6}
    \item The reference signal is $\zreff$ is constant and $S(\zreff)=0$.
    \label{rel:4}
    \item $S$ is locally linear near $\zreff$, that is, there exists $\eps>0$ such that $S(z)=\frac{\dd S}{\dd z}(\zreff) (z-\zreff)$ for all $z\in\R^n$ satisfying $|z-\zreff|\leq\eps$. Moreover, $\frac{\dd S}{\dd z}(\zreff)$ is invertible.
    \label{rel:7}
\end{enumerate}
Under these restrictions, we now consider the following problem.
\begin{problem}\label{pb:sim}
    Consider the system~\eqref{eq:wc_discrete}.
    From the knowledge of $S$, $\tau$ and $d$ and the online measurement of $u(t)$ and $z(t)$,
    for any arbitrary small neighborhood of $0\in\FF_{z}$,
    find $u$ in the form of a dynamic output feedback law that stabilizes $z$ in this neighborhood, and estimate online $w$.
\end{problem}

As explained in Remark~\ref{rem:noconv}, there is no hope of estimating $w$ when employing the feedback law \eqref{eq:cont}, since stabilizing $z$ at $\zreff$ prevents the persistency condition from being satisfied. For this reason, we suggest adding to the control law a small excitatory signal, whose role is to improve persistency without perturbing too much the $z$ dynamics. Of course, this strategy prevents obtaining exact stabilization. This is why this condition has been relaxed in \ref{rel:1}.
More precisely, to answer Problem~\ref{pb:sim}, we suggest to consider the feedback law
\begin{align}
        u(t, r) = &v(t,r) - \alpha (z(t, r)-\zreff) + z(t, r)\nonumber\\
        &- \int_{r'\in\Omega}\what(t, r, r')S(z(t-\dbar, r'))\dd r',\label{eq:cont_sim}
\end{align}
coupled with the observer
\begin{equation}\label{eq:simobs}
\begin{aligned}
        &\tau\frac{\partial \zhat}{\partial t}(t, r) =  - \alpha\zhat(t,r) + v(t,r)
        \\
        &\tau\frac{\partial \what}{\partial t}(t, r, r') = -(\zhat(t, r)-z(t,r))S(z(t-\dbar, r'))^\top
\end{aligned}
\end{equation}
where $\alpha>0$ is a tunable controller gain
and $v$ is a signal to be chosen both
small enough in order to ensure practical convergence of $z$ towards $0$
and
persistent in order to ensure convergence of $\what$ towards $w$. 
In practice, $v$ can also be seen as an external signal arising from interaction with other neurons whose dynamics are not modeled.

Since $v$ is supposed to be continuous, the proof of well-posedness is identical to the proof of Proposition~\ref{th:wp_closed} and we get the following.

\begin{proposition}[Well-posedness of \eqref{eq:wc}-\eqref{eq:cont_sim}-\eqref{eq:simobs}]\label{th:wp_closed_sim}
Suppose that Assumption~\ref{ass:wp} is satisfied.
For any $v\in C^0(\R_+, \FF_{z})$ and any initial condition $(z_{0}, \zhat_{0}, \what_{0})\in C^0([-\dbar, 0], \FF_{z})^2\times\FF_{w}$,
the closed-loop system \eqref{eq:wc}-\eqref{eq:cont_sim}-\eqref{eq:simobs} admits a unique corresponding solution $(z, \zhat, \what)\in C^1([0, +\infty), \FF_{z}^2\times\FF_{w})\cap C^0([-\dbar, +\infty), \FF_{z}^2\times\FF_{w})$.
\end{proposition}

\begin{theorem}[Simultaneous estimation and practical stabilization]\label{th:sim}
Suppose that Assumption~\ref{ass:wp} is satisfied.
Then, for all $\alpha>0$ and all input $v\in C^0(\R_+,\FF_{z})$,
any solution of \eqref{eq:wcij}-\eqref{eq:cont_sim}-\eqref{eq:simobs} is such that
$$\limsup_{t\to+\infty}\|z(t)\|_{\FF_{z}}
\leq \limsup_{t\to+\infty}\|v(t)\|_{\FF_{z}}
$$
and $\|\what(t)\|_{(\FF_{w}, \|\cdot\|_F)}$ remains bounded for all $t\geq0$.
\smallskip

Moreover, there exists $\eps>0$ such that if
$v\in C^1(\R_+,\FF)$ is bounded by $\eps$, has bounded derivative, and is PE with respect to $\Id_\FF$, then we also have
\begin{equation*}
    \lim_{t\to+\infty} \|\what(t)-w\|_{(\FF_{w}, \|\cdot\|_F)}=0.
\end{equation*}
\end{theorem}

The proof of Theorem~\ref{th:sim} relies on the following lemma that states standard properties of the PE condition.

\begin{lemma}[Properties of PE]\label{lem:PE0}
Let $\FF$  be a Hilbert space and $\YY$ be a Banach space.
Let $g\in C^0(\R_+, \FF)$ and $P\in\LL(\FF, \YY)$.
\begin{enumerate}[label = \textit{(\alph*)}]
    \item If \eqref{eq:pe} is satisfied only for $t\geq t_0$ for some $t_0\geq0$, then $g$ is PE with respect to $P$.
    \label{i1}
    \item If $g$ is PE with respect to $P$, then $t\mapsto g(t-\dbar)$ is also PE with respect to $P$ for any positive delay $\dbar\geq 0$.
    \label{i0}
    \item If $g$ is PE with respect to $P$ and $W\in\LL(\FF)$, then $Wg$ is PE with respect to $PW^*$.
    \label{i3}
\end{enumerate}
Moreover, if $\FF=\YY$ is finite-dimensional and $P=\Id_\FF$, we also have:
\begin{enumerate}[resume,label = \textit{(\alph*)}]
    \item If $g$ is PE with respect to $\Id_\FF$ with constants $T$ and $\kappa$ in \eqref{eq:pe}, and $g$ is bounded by some positive constant $M$, then $M\geq\sqrt{\frac{\kappa}{T}}$.
    Conversely, for all positive constants $T$, $\kappa$ and $M$ such that $M\geq\sqrt{\frac{2\kappa\dim\FF}{T}}$, there exists $g\in C^1(\R_+, \FF)$ that is bounded by $M$
    and PE with respect to $\Id_\FF$ with constants ($T$, $\kappa$).
    \label{i4}
    \item If $g$ is bounded and PE with respect to $\Id_\FF$ and $\delta \in C^0(\R_+, \FF)$ is such that $\delta(t)\to0$ as $t\to+\infty$, then $g+\delta$ is also PE with respect to $\Id_\FF$.
    \label{i5}
    \item 
    If $\mu$ is a positive constant and $g\in C^1(\R_+,\FF)$ is bounded, with bounded derivative, and PE with respect to $\Id_\FF$, then
    any solution $z\in C^1(\R_+, \FF)$ of $\frac{\dd z}{\dd t} = -\mu z(t) + g(t)$ is also PE with respect to $\Id_\FF$.
    \label{i6}
\end{enumerate}
\end{lemma}

Lemma~\ref{lem:PE0} is proved in \ref{app:proof}.

\begin{proof}[Proof of Theorem~\ref{th:sim}]
Clearly, from \eqref{eq:simobs} and Grönwall's inequality, $\limsup_{t\to+\infty}$ $\|\zhat(t)\|_{\FF_{z}}
\leq \limsup_{t\to+\infty}\|v(t)\|_{\FF_{z}}$.
Moreover, setting $\zeps_1 = \zhat-z$ and $\weps_{11} = \what-w$, we see that, due to the choice of $u$ given by \eqref{eq:cont_sim}, $(\zeps_1, \weps_{11})$ satisfies \eqref{eq:eps}. Hence, according to Theorem~\ref{th:obs},
$\limsup_{t\to+\infty}\|\zhat(t)-z(t)\|_{\FF_{z}} = 0
$ and $\what(t)$ remains bounded. This yields the first part of the result.
To show the second part of the result, it is sufficient to show that $t\mapsto (r\mapsto S(z(t-\dbar, r))$ is PE with respect to $\Id_{\FF_z}$ in order to apply the second part of Theorem~\ref{th:obs}.

To end the proof of Theorem~\ref{th:sim}, we use Lemma~\ref{lem:PE0} as follows.
Using Assumption~\ref{rel:7}, let $\eps>0$ be such that
$S(z)=\frac{\dd S}{\dd z}(\zreff)(z-\zreff)$ for all $z\in\FF$ satisfying $|z-\zreff|\leq\frac{2\eps}{\alpha}$.
Assume that
$v\in C^1(\R_+,\FF)$ is bounded by $\eps$, has bounded derivative, and is PE with respect to $\Id_\FF$.  Such a signal $v$ exists by \ref{i4}.
By \ref{i3}, $\tau^{-1}v$ is also PE with respect to $\Id_\FF$.
Since $\zhat-\zreff$ satisfies \eqref{eq:simobs},
\ref{i6} with $\mu=\tau^{-1}\alpha$  shows that
$\zhat-\zreff$ is also PE with respect to $\Id_\FF$.
Moreover, by Grönwall's inequality, there exists $t_0\geq 0$ such that $|\zhat(t, r_k)-\zreff|\leq \frac{2\eps}{\alpha}$ for all $t\geq t_0$ and all $k\in\{1,\dots,N\}$. Then,
$S(\zhat(t, r_k)) = \frac{\dd S}{\dd z}(\zreff)(\zhat(r_k)-\zreff)$. Hence,
\ref{i1} and \ref{i3} ensure that $t\mapsto S(\zhat(t))$ is PE with respect to $\frac{\dd S}{\dd z}(\zreff)^\top$.
Since $\frac{\dd S}{\dd z}(\zreff)$ is invertible, $t\mapsto S(\zhat(t))$ is PE with respect to $\Id_\FF$.
Since $\zhat(t)-z(t)\to 0$ as $t\to+\infty$ and $S$ has bounded derivative, $S(\zhat(t)) - S(z(t))\to0$.
Hence, according to \ref{i5}, $t\mapsto S(z(t))$ is also PE with respect to $\Id_\FF$.
Hence, by \ref{i0}, $t\mapsto S(z(t-\dbar))$ is PE with respect to $\Id_{\FF_z}$, which concludes the proof of Theorem~\ref{th:sim}.
\end{proof}

\begin{remark}\label{rem:rest}
    All assumptions~\ref{rel:1}-\ref{rel:7} have been used in the second part of the proof of Theorem~\ref{th:sim} at crucial points where, without them, one cannot conclude. In particular, these assumptions allow to use the properties stated in Lemma~\ref{lem:PE0}. Without them, stronger versions of these properties should be required to show that $S(\zhat(\cdot-d))$ is PE with respect to $\Id_{\FF_z}$. For example, without \ref{rel:3}, the property \ref{i5} would be required in an infinite-dimensional context, which is impossible since counter-examples can easily be found.
    Without \ref{rel:6}, \ref{i6} would be required for filters of the form $\frac{\dd z}{\dd t} = -\Sigma z(t) + g(t)$ where $\Sigma$ is a positive definite matrix, which is also known to be false (see \cite[Example 7]{doi:10.1080/00207178708933715}). Similarly, \ref{rel:2} is required to have that $\zhat$ is a filter of $v$ in the form of \eqref{eq:simobs}, which is necessary to use \ref{i6}.
    Without \ref{rel:5}, \ref{i0} would be required for non-constants delays, which is not possible due to counter-examples such as $\R_+\ni t\mapsto (\sin(t-d_1), \cos(t-d_2))\in\R^2$ with $d_1 = 0$ and $d_2 = \frac{\pi}{2}$.
    Properties \ref{rel:4} and \ref{rel:7} are used in the end of the proof of Theorem~\ref{th:sim}. Without them, passing from $\zhat-\zreff$ being PE to $S(\zhat)$ being PE remains an open problem.
\end{remark}

\section{Numerical simulations}\label{sec:num}

We provide numerical simulations of the observer and controllers proposed in Theorems~\ref{th:obs}, \ref{th:stab}, and \ref{th:sim}. The observer simulations are in line with those presented in \cite{brivadis:hal-03660185}, while the two controllers (exact stabilization and simultaneous practical stabilization and estimation) are new. Although disturbances could not be accounted for in our theoretical developments, we also assess robustness of the developed control law to  perturbations.

We consider the case of a two-dimensional neural field (namely, $n_1=n_2 =1$) over the unit circle $\Omega =\mathbb{S}^1$ with constant delay $\dbar$.
The kernels are given by Gaussian functions depending on the distance between $r$ and $r'$, as it is frequently assumed in practice (see \cite{DECH16}):
$w_{ij}(r, r') = \omega_{ij}\gfrak(r, r')/\|\gfrak\|_{L^2(\Omega^2; \R)}$, with $\gfrak(r, r') = \exp(-\sigma|r-r'|^2)$ for constant parameters $\sigma$ and $\omega_{ij}$ given in Table~\ref{tab:param}.
Simulations code
can be found in repository \cite{git}. The system is spatially discretized over $\Omega$ with a constant space step $\Delta r = 1/20$, and the resulting delay differential equation is solved with an explicit Runge-Kutta $(2,3)$ method.
Initial conditions are taken as
$z_1(0, r) = z_2(0, r) = \zhat_1(0, r) = 1$, $\zhat_2(0, r) = \hat w_{11}(0, r, r') = \hat w_{12}(0, r, r')=0$ for all $r, r'\in\Omega$.

In order to test the observer \eqref{eq:obs},
the inputs $u_i$ are chosen as spatiotemporal periodic signals with irrational frequency ratio, i.e., $u_i(t, r) = \mu\sin(\lambda_i tr)$ with $\mu=10^3$ and $\lambda_1/\lambda_2$ irrational.
This choice is made to ensure persistency of excitation of the input $(u_1, u_2)$, which in practice seems to be sufficient to induce persistency of excitation of $t\mapsto(S_{11}(z_1(t-\dbar)), S_{12}(z_2(t-\dbar)))$. Note that for $u_1=u_2=0$, the persistency of excitation assumption seems to be not guaranteed. Hence the observer does not converge
(the plot is not reported).
For testing the controllers, the inputs are respectively chosen as \eqref{eq:cont} for exact stabilization and \eqref{eq:cont_sim} with $v(t,r)= \mu \sin(\lambda_1tr)$, $\mu\in\R$, for simultaneous practical stabilization and estimation. In the latter case, we must fix $n_2=0$ as imposed by Section~\ref{sec:sim}~\ref{rel:2}.

The parameters  of the system \eqref{eq:wcij}, the observer \eqref{eq:obs}, and the controller \eqref{eq:cont} are set as in Table~\ref{tab:param}, so that Assumptions~\ref{ass:wp} and \ref{ass:diss} are fulfilled.
The convergence of the observer error \eqref{eq:eps} towards zero is verified in Figure~\ref{fig:obs}. In particular, the estimation of $w_{11}$ by the observer is shown at several time steps in Figure~\ref{fig:film}.
The convergence of the state towards zero ensured by the controller \eqref{eq:cont} is shown in Figure \ref{fig:stab}.
As explained in Remark~\ref{rem:noconv}, no convergence of the kernels estimation can be hoped for in Figure~\ref{fig:stab} since stabilizing the state prevents PE.

Figure \ref{fig:sim} enlightens the compromise made by the feedback law \eqref{eq:cont_sim} between observation and estimation: by choosing $v(t,r)= \mu \sin(\lambda_1tr)$ with $\mu=100$, the asymptotic regime of the state remains in a neighborhood of zero (practical stabilization), which allows the convergence of the kernel estimation $\what$ towards $w$. When increasing $\mu$, the estimation rate increases (one obtains a plot similar to Figure~\ref{fig:obs} for $\mu=10^3$) but the asymptotic regime of the state moves away from zero.
On the contrary, when decreasing $\mu$ towards zero, one obtains an asymptotic regime of $z$ closer to zero (one obtains a plot similar to Figure~\ref{fig:stab} for $\mu=0.1$), but the convergence of $\what$ towards $w$ is slower.

In order to numerically assess robustness to uncertainties or disturbances, we have considered a situation in which a constant and spatially uniform perturbation is added to the control input (in the case where the neuronal population is fully actuated, namely $n_2 = \dim z_2=0$). Figure \ref{fig-pert} reports the steady-state behavior of the $L^2$-norm of $z_1$ as a function of the value of the applied perturbation. Two indicators are used to that aim: the $\limsup$ of $\|z_1(t)\|_{L^2}$ as $t\to+\infty$ (computed here as the maximal value reached by $\|z_1(t)\|_{L^2}$ over the time interval $[5,10]$) and the steady-state average of $\|z_1(t)\|_{L^2}$ (computed here as $\frac{1}{5}\int_5^{10} \|z_1(t)\|_{L^2}\dd t$). We see that, although not guaranteed by our theoretical results, the proposed output feedback control law seems to exhibit some robustness to actuation perturbations, as the steady-state value of $\|z_1(t)\|_{L^2}$ remains small for sufficiently small perturbations. Interestingly, for a perturbation above $10$, some steady-state oscillations take place, which explains why the steady-state value of the average of $\|z_1(t)\|_{L^2}$ (red dots) becomes significantly lower than that of its maximal value (blue dots).

To theoretically show this result, a natural property to require is the input-to-state stability (ISS) (see. e.g., \cite{Mironchenko2023}). Roughly speaking, it requires that, in the presence of an approximately known delay, disturbance, or noise, the state still converges to a neighborhood of the target point, whose size tends towards 0 as the uncertainty's amplitude tends towards 0. Going from asymptotic stability to ISS is not an easy task in general. The most common way is to derive an ISS-Lyapunov function from the original Lyapunov function used in the paper. This path has been followed in \cite{DECH16} for example. However, it does not work in our case because our Lyapunov function is non-strict: its derivative along the system's solutions involves negative terms in only part of the state variables (see inequality~\eqref{eq:conv1}). Therefore, being able to derive ISS remains an open question that would probably require a strictification technique and potentially a modification of our algorithm.
    
Despite the robustness suggested by Figure \ref{fig-pert}, it is worth mentioning that the adaptive observer considered here is not free from a possible parameter drift. For a constant perturbation of amplitude $2$, Figure \ref{fig-drift} shows that the $L^2$ norm of the estimation error $\tilde w$ diverges. This well-known phenomenon in adaptive control \cite{IOANNOU1984583} raises practical implementability issues as $\hat w$ is involved in the control law. Nevertheless, interestingly, the right graph of Figure \ref{fig-drift} suggests that, despite this parameter drift, the control law (18) remains bounded over time.


\begin{table}[ht!]
    \centering
    \vspace{0.25cm}
    \begin{tabular}{|c|c|c|c|}
        \hline
        $S_{ij} = \tanh$&
        $\tau_i = 1$ &
        $\lambda_1 = 100$ &
        $\lambda_2 = 100\sqrt{2}$\\
        \hline
        $\alpha = 100$&
        $\dbar = 0.1$
        &
        $\zref{i} = 0$&
        $\sigma = 60$\\
        \hline
        $\omega_{11} = 2$&
        $\omega_{12} = 2$&
        $\omega_{21} = -2$&
        $\omega_{22} = 0.1$\\
        \hline
    \end{tabular}
    \caption{System and observer parameters for the numerical simulation of Figures~\ref{fig:obs}--\ref{fig:sim}}
    \label{tab:param}
\end{table}

\begin{figure}[ht!]
    \centering
    \begin{subfigure}[b]{0.49\linewidth}
         \centering
         \includegraphics[width=\linewidth]{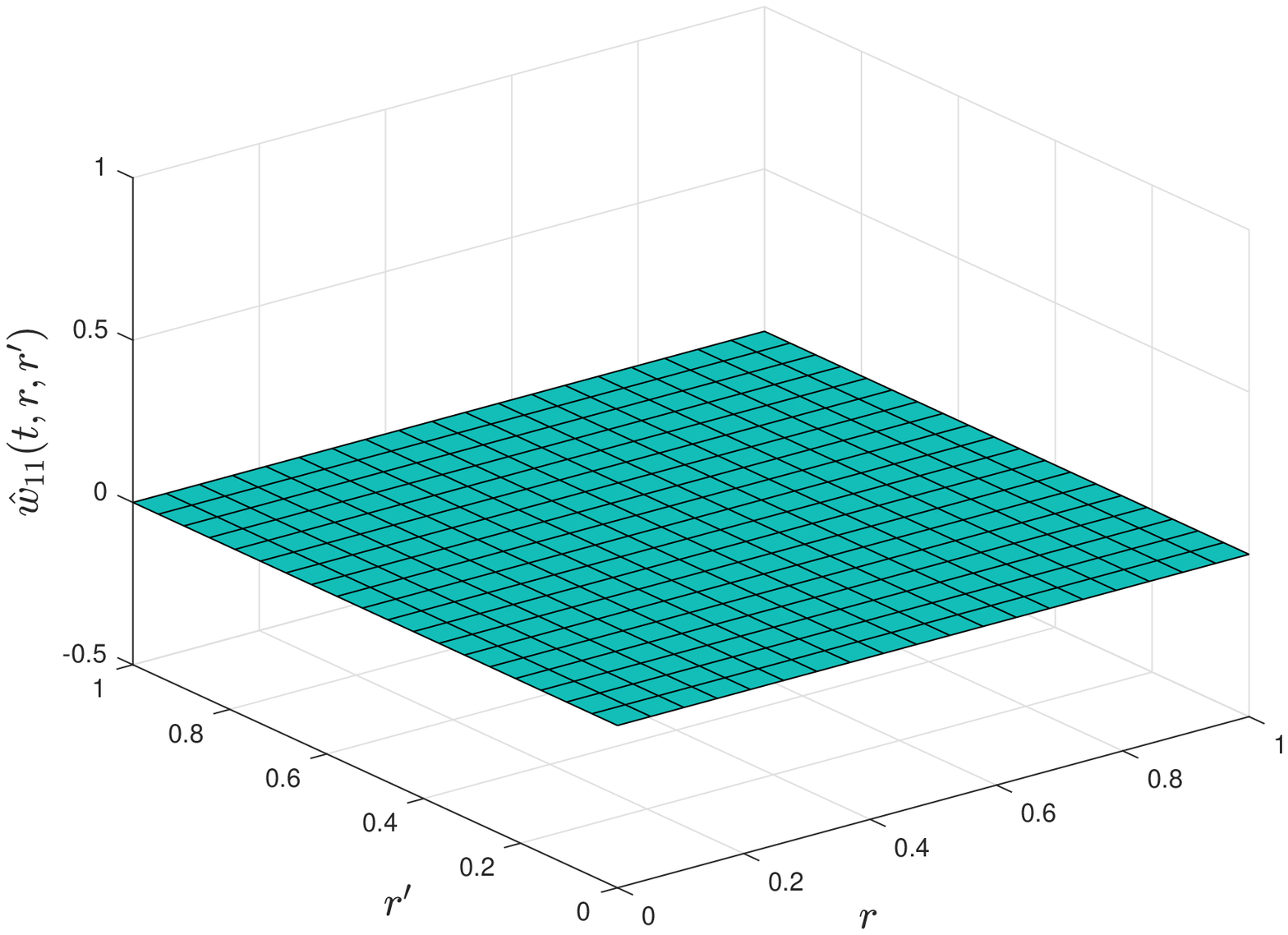}
         \caption{$t = 0$}
     \end{subfigure}
     \hfill
     \begin{subfigure}[b]{0.49\linewidth}
         \centering
         \includegraphics[width=\linewidth]{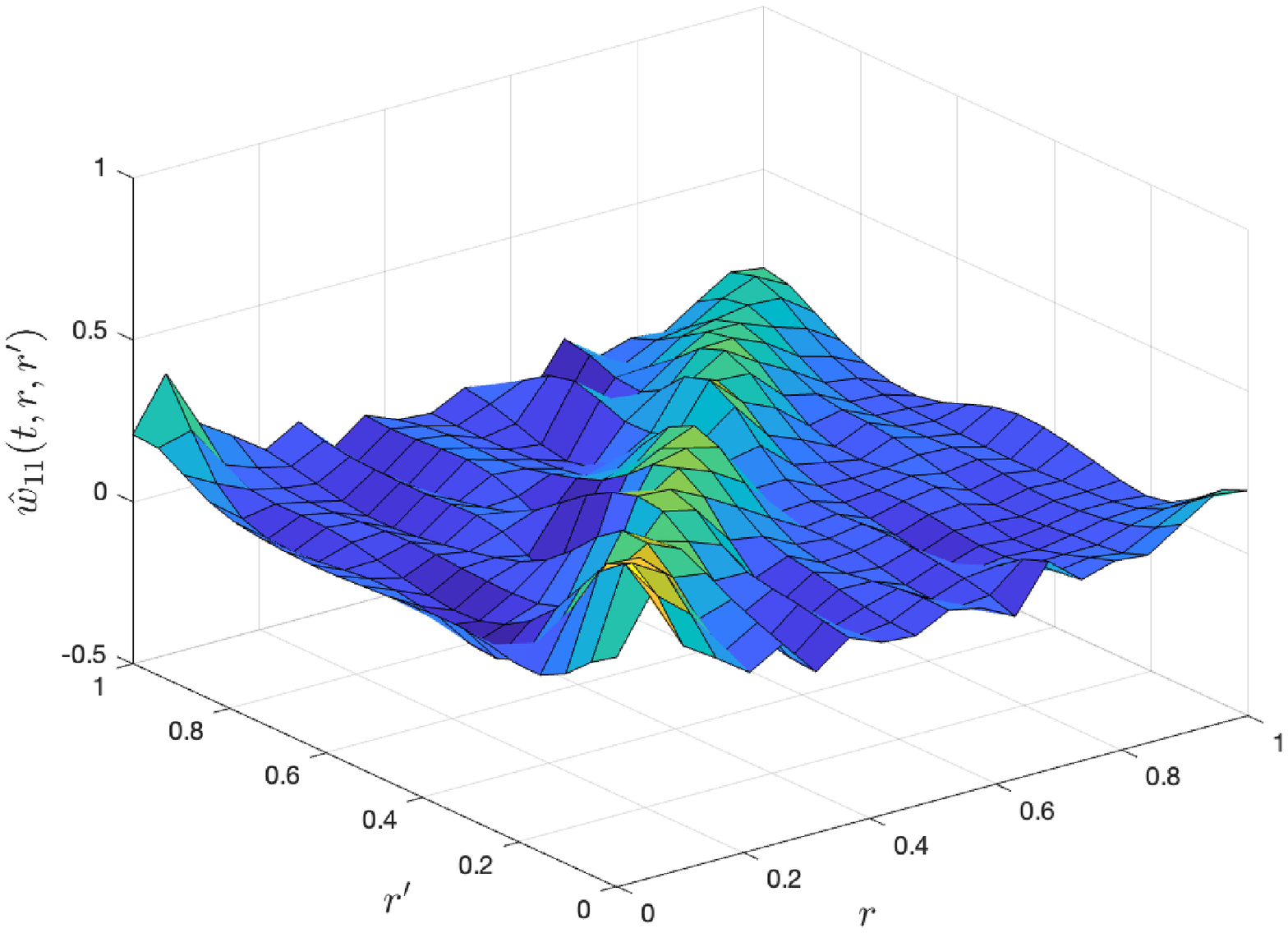}
         \caption{$t = 2$}
     \end{subfigure}
     \\
     \begin{subfigure}[b]{0.49\linewidth}
         \centering
         \includegraphics[width=\linewidth]{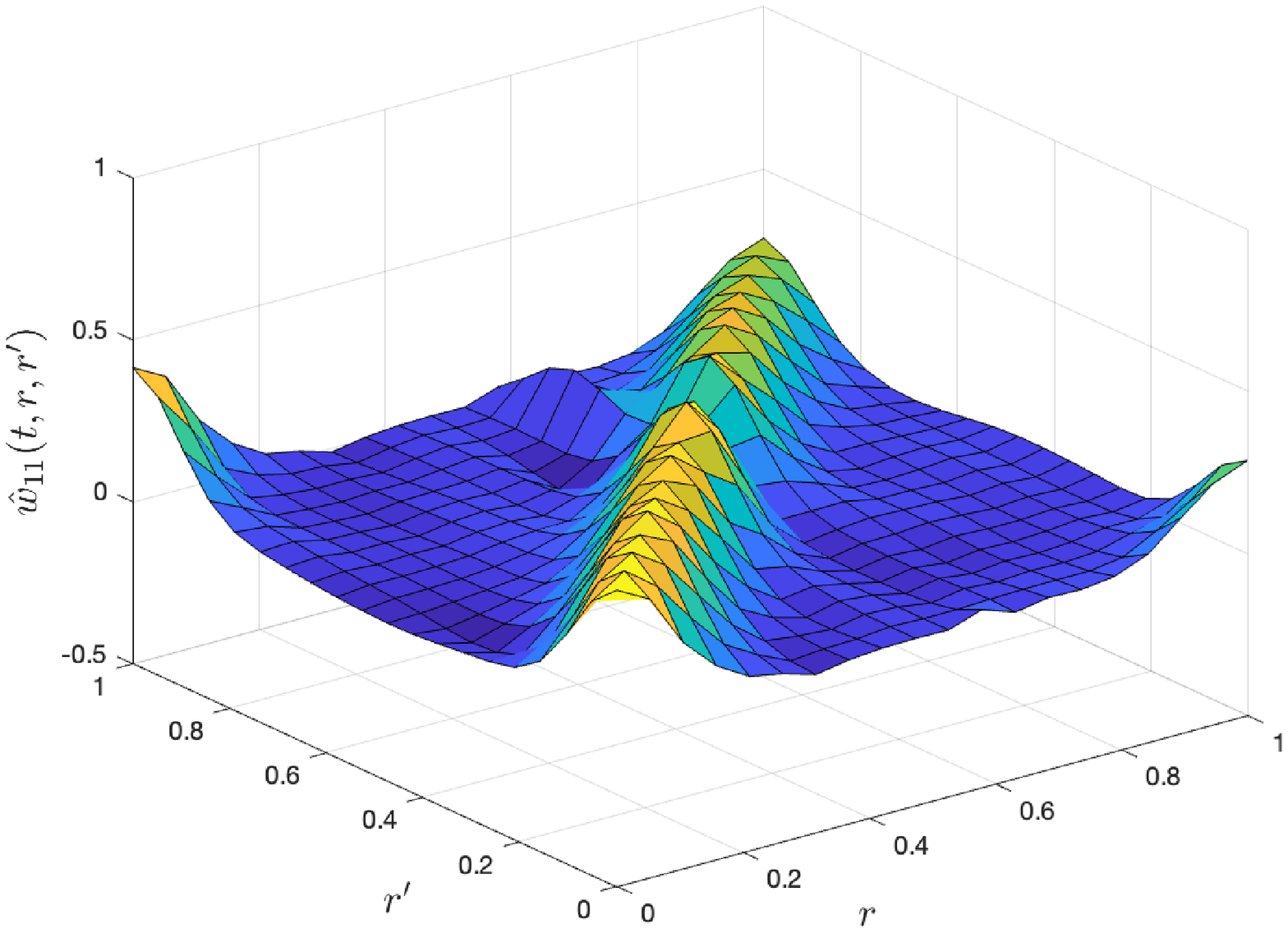}
         \caption{$t = 5$}
     \end{subfigure}
     \hfill
     \begin{subfigure}[b]{0.49\linewidth}
         \centering
         \includegraphics[width=\linewidth]{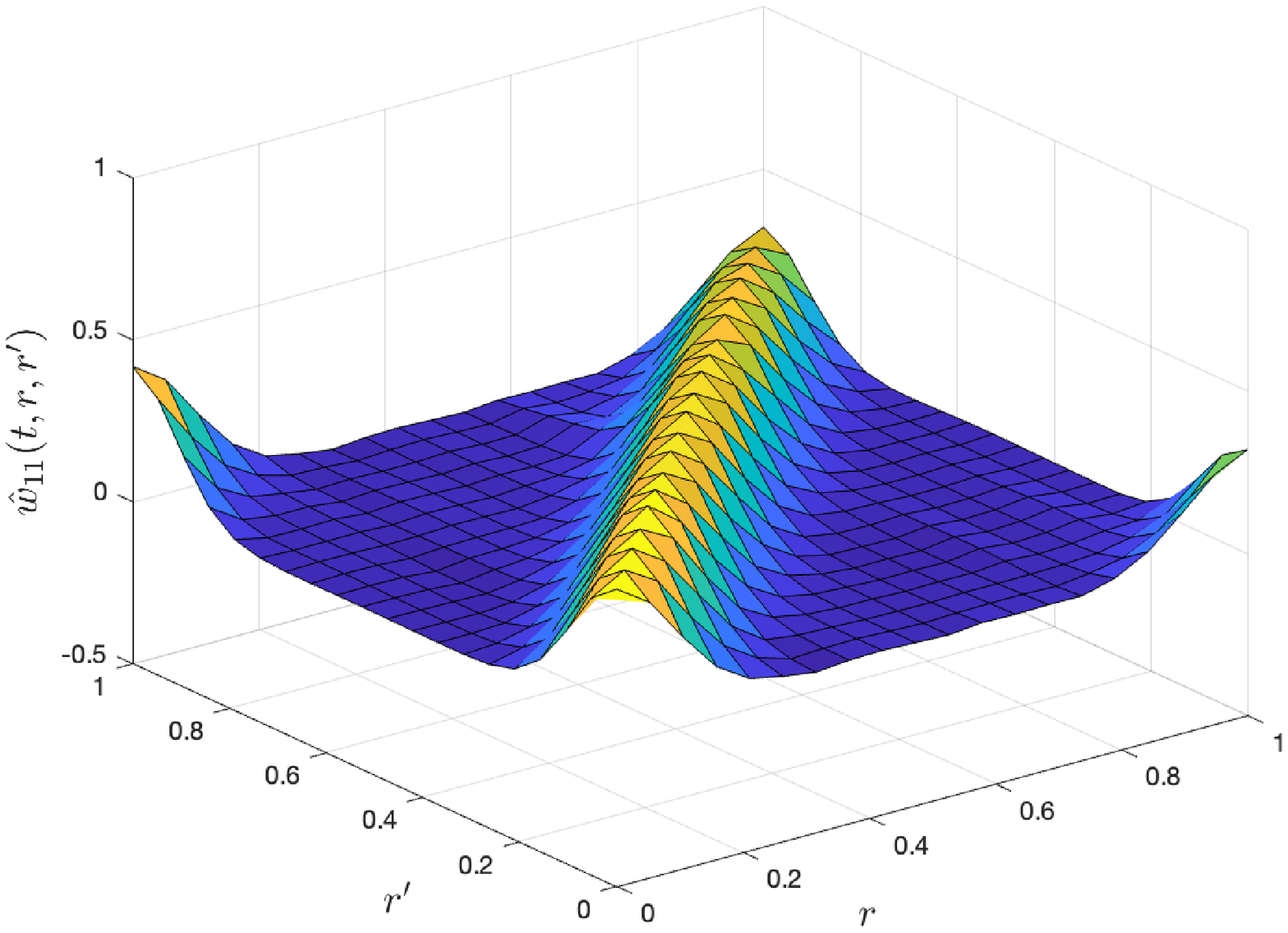}
         \caption{$t = 10$}
     \end{subfigure}
    \caption{
    Evolution of the kernel estimation $\hat w_{11}(t, r, r')$ when running the observer \eqref{eq:obs}.
    }
    \label{fig:film}
\end{figure}

\begin{figure}[ht!]
    \centering
    \includegraphics[width=0.49\linewidth]{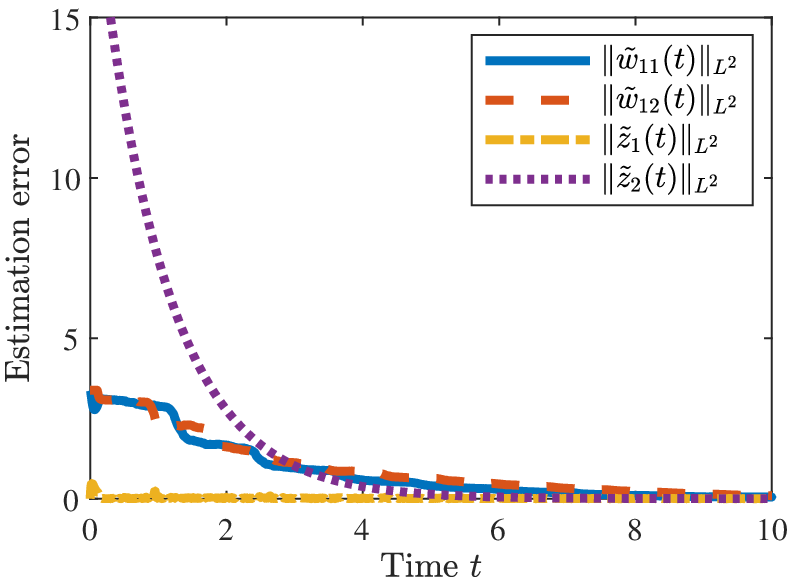}
    \caption{
    Evolution of the estimation errors 
$\|\weps_{1i}\|_{\FF_{w_{1i}}}$ and
$\|\zeps_i\|_{\FF_{z_i}}$
for $i\in\{1,2\}$
of the observer \eqref{eq:obs}.
}
\label{fig:obs}
\end{figure}

\begin{figure}[ht!]
    \centering
    \begin{subfigure}[b]{0.49\linewidth}
         \centering
         \includegraphics[width=\linewidth]{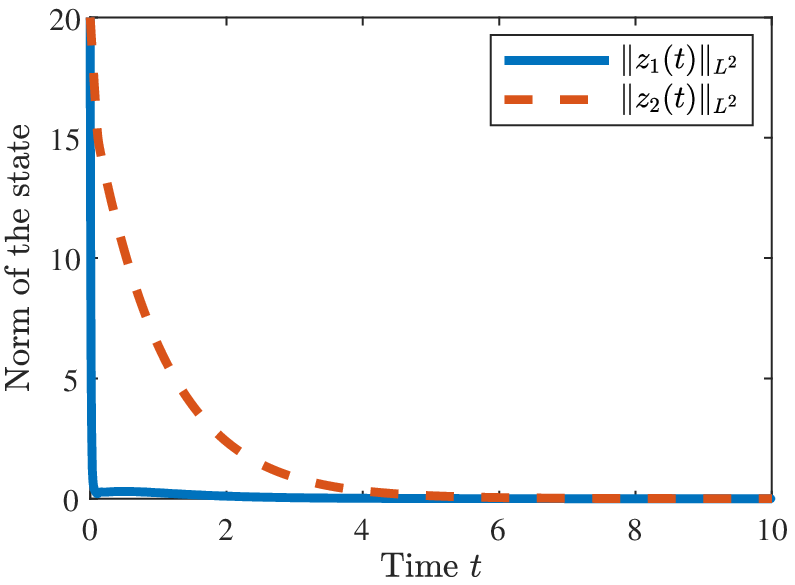}
     \end{subfigure}
     \hfill
    \begin{subfigure}[b]{0.49\linewidth}
         \centering         \includegraphics[width=\linewidth]{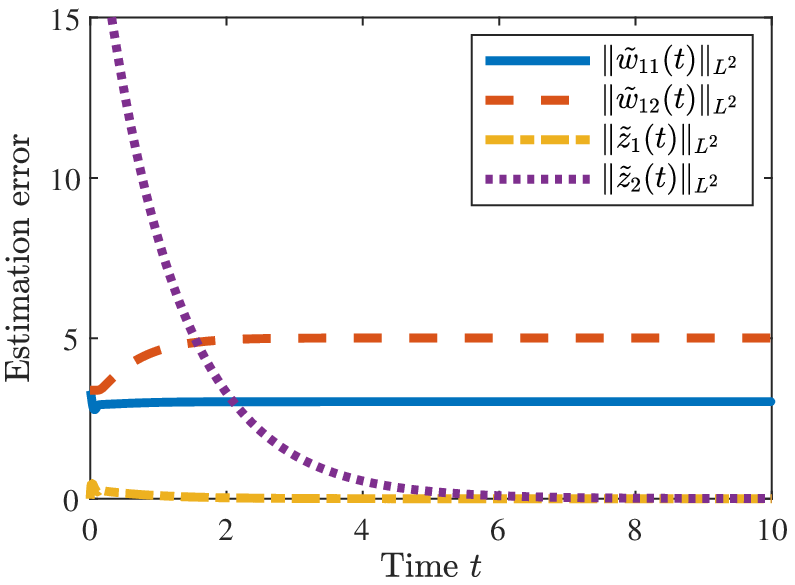}
     \end{subfigure}
     \\
     \begin{subfigure}[b]{0.49\linewidth}
         \centering
         \includegraphics[width=\linewidth]{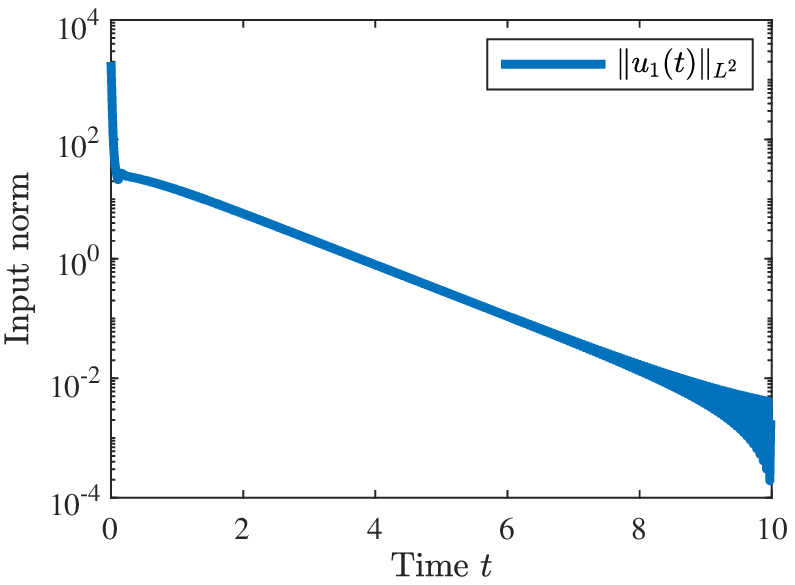}
     \end{subfigure}
    \caption{
    Evolution of the norm of the state $z_i$, of the estimation errors 
$\weps_{1i}$ and
$\zeps_i$
for $i\in\{1,2\}$, and of input $u_1$,
for the control law \eqref{eq:cont}.
}
\label{fig:stab}
\end{figure}

\begin{figure}[ht!]
    \centering
    \begin{subfigure}[b]{0.49\linewidth}
         \centering
         \includegraphics[width=\linewidth]{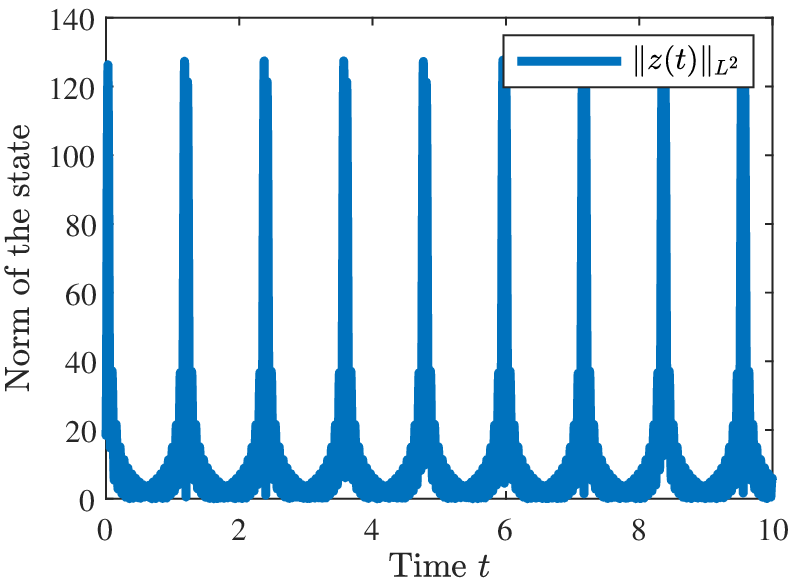}
     \end{subfigure}
     \hfill
     \begin{subfigure}[b]{0.49\linewidth}
         \centering
         \includegraphics[width=\linewidth]{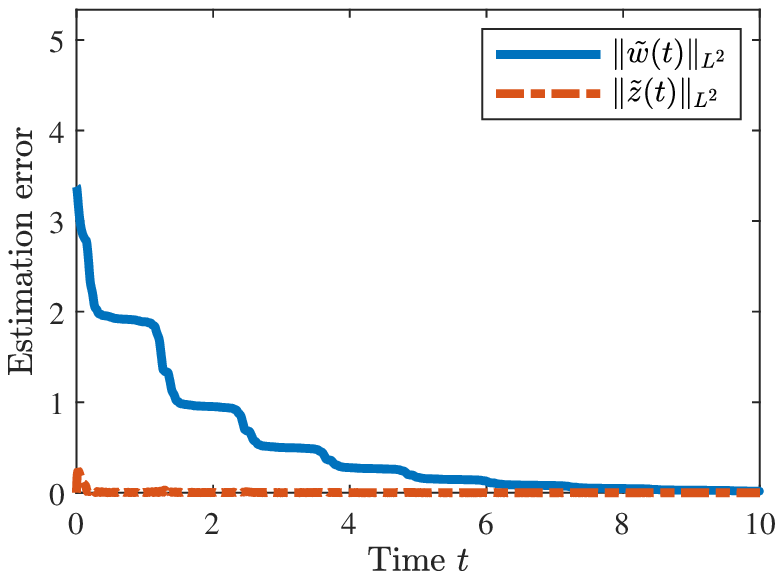}
     \end{subfigure}
    \caption{
    Evolution of the norm of the state $\|z\|_{\FF_{z}}$ and of the estimation errors 
$\|\weps\|_{\FF_{w}}$ and
$\|\zeps\|_{\FF_{z}}$
for the control law \eqref{eq:cont_sim} with $v(t, r)=100\sin(\lambda_1 tr)$.
}
\label{fig:sim}
\end{figure}

\begin{figure}[ht!]
    \centering
    \includegraphics[width=0.49\linewidth]{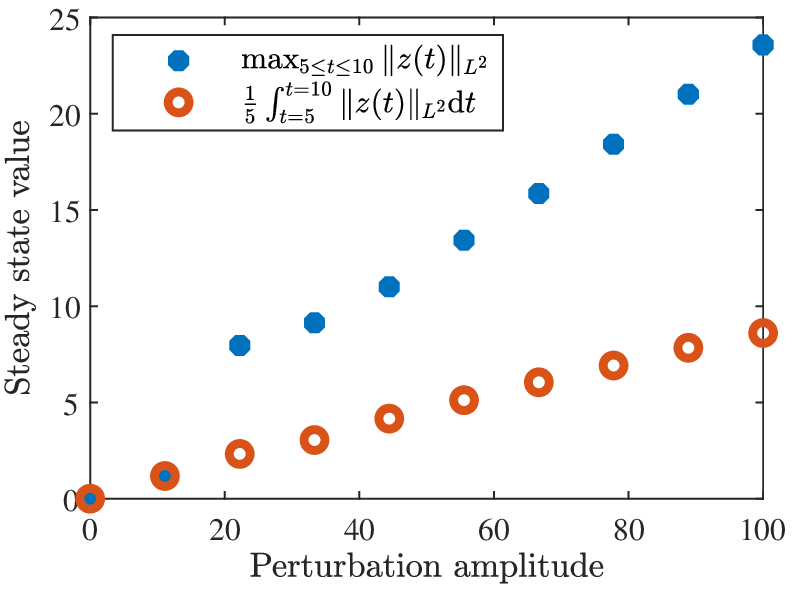}
    \caption{Steady-state behavior of $\|z_1(t)\|_{L^2}$, in terms of its maximal (blue dots) and average (red dots) values for the control law \eqref{eq:cont}, as a function of the value of the applied additive disturbance.
}
\label{fig-pert}
\end{figure}

\begin{figure}[ht!]
    \centering
    \begin{subfigure}[b]{0.49\linewidth}
         \centering
         \includegraphics[width=\linewidth]{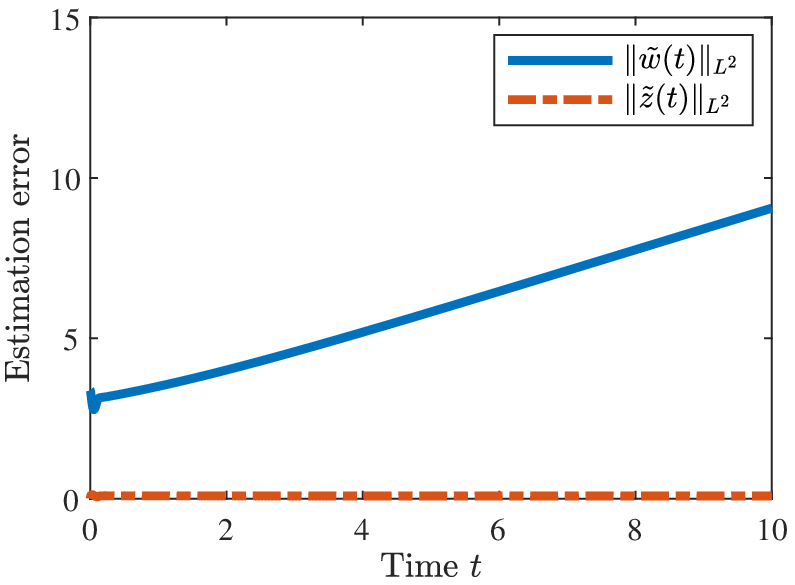}
     \end{subfigure}
     \hfill
    \begin{subfigure}[b]{0.49\linewidth}
         \centering         \includegraphics[width=\linewidth]{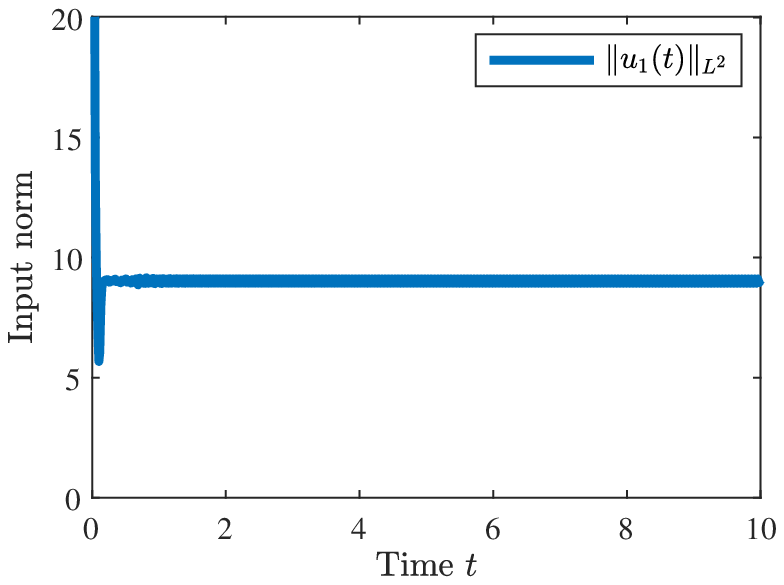}
     \end{subfigure}
    \caption{
    Evolution of the estimations errors $\weps$ and $\zeps$ (left) and of the control input $u_1$ (right)
    for \eqref{eq:cont} in the presence of an additive perturbation of amplitude $2$.
}
\label{fig-drift}
\end{figure}

\section{Conclusion}

In this paper, a new adaptive observer has been proposed to estimate online the synaptic strength between neurons from partial measurement of the neuronal activity. We proved the convergence of the observer under a persistency of excitation condition by designing a Lyapunov functional taking into account the infinite-dimensional nature of the state due to the spatial distribution of the neuronal activity and to the time-delay. We have shown that this observer can be used to design dynamic feedback laws that stabilize the system to a target point, even without persistency of excitation.
From the theoretical viewpoint, the main open question remains to extend our result on simultaneous estimation and stabilization. It currently relies on important limitations on the system, that cannot be lifted without a deeper analysis of the PE condition proposed in the paper. In particular, sufficient conditions ensuring that choosing a PE input signal guarantees PE of the state of the neural fields should be sought.

\appendix

\section{Proof of Lemma~\ref{lem:PE0}}\label{app:proof}

\begin{enumerate}[label = \textit{(\alph*)}]
\item If, for all $t\geq t_0$,
    $
    \int_t^{t+T}|\langle g(\tau), x\rangle_\FF|^2\dd\tau \geq \kappa \|Px\|_\YY^2$ for all $x\in\FF$,
    then for all $t\geq0$,
    \begin{align*}
    \int_t^{t+t_0+T}|\langle g(\tau), x\rangle_\FF|^2\dd\tau
    &\geq \int_{t+t_0}^{t+t_0+T}|\langle g(\tau), x\rangle_\FF|^2\dd\tau\\
    &\geq \kappa \|Px\|_\YY^2,\quad \forall x\in\FF,
    \end{align*}
    which shows that $g$ is PE with respect to $P$.
\item If, for all $t\geq 0$,
    $
    \int_t^{t+T}|\langle g(\tau), x\rangle_\FF|^2\dd\tau \geq \kappa \|Px\|_\YY^2$ for all $x\in\FF$,
    then for all $t\geq \dbar$,
    \begin{align*}
    \int_{t}^{t+T}|\langle g(\tau-\dbar), x\rangle_\FF|^2\dd\tau
    &= \int_{t-\dbar}^{t+T-\dbar}|\langle g(\tau), x\rangle_\FF|^2\dd\tau \\
   & \geq \kappa \|Px\|_\YY^2,\quad \forall x\in\FF,
    \end{align*}
    which shows that $t\mapsto g(t-\dbar)$ is PE with respect to $P$ by \ref{i1}.
\item If, for all $t\geq 0$,
    $
    \int_t^{t+T}|\langle g(\tau), x\rangle_\FF|^2\dd\tau \geq \kappa \|Px\|_\YY^2$ for all $x\in\FF$,
    then for all $t\geq 0$,
    \begin{align*}
    \int_{t}^{t+T}|\langle Wg(\tau), x\rangle_\FF|^2\dd\tau
    &= \int_{t}^{t+T}|\langle g(\tau), W^*x\rangle_\FF|^2\dd\tau \\
    &\geq \kappa \|PW^*x\|_\YY^2,\quad \forall x\in\FF,
    \end{align*}
    which shows that $Wg$ is PE with respect to $PW^*$ by \ref{i1}.
\item If $g$ is bounded by $M$ and for all $t\geq 0$,
    $
    \int_t^{t+T}|\langle g(\tau), x\rangle_\FF|^2\dd\tau \geq \kappa \|x\|_\FF^2$ for all $x\in\FF$,
then Cauchy-Schwartz inequality yields
$\int_t^{t+T}|\langle g(\tau), x\rangle_\FF|^2\dd\tau\leq TM^2\|x\|^2_\FF$, hence $M\geq\sqrt{\frac{\kappa}{T}}$.
Conversely, set $g(\tau) = \sqrt{\frac{2\kappa}{T}}\sum_{\ell=1}^{\dim\FF}\sin(\frac{2\ell\pi\tau}{T})e_\ell$, where $e$ is a basis of $\XX$. Then $g$ is bounded by $\sqrt{\frac{2\kappa\dim\FF}{T}}$ and has a bounded derivative. Moreover, for all $t\geq0$ and all $x=\sum_{\ell=1}^{\dim\FF}x_\ell e_\ell\in\FF$,
\begin{align*}
    &\int_t^{t+T}|\langle g(\tau), x\rangle_\FF|^2\dd\tau
    =
    \frac{2\kappa}{T}\int_t^{t+T}|\sum_{\ell=1}^{\dim\FF} x_\ell 
    \\
    &\sin\left(\frac{2\ell\pi\tau}{T}\right)|^2\dd\tau=\frac{2\kappa}{T}\sum_{\ell=1}^{\dim\FF} x_\ell^2 \int_0^{T}\sin^2\left(\frac{2\ell\pi\tau}{T}\right)\dd\tau
    \\
    &=\kappa \|x\|^2_\FF.
\end{align*}
Hence, $g$ is PE with respect to $\Id_\FF$ with constants $T$ and $\kappa$.

\item Denote by $M$ a bound of $g$. Denote by $T$ and $\kappa$ the PE constants of $g$ with respect to $\Id_\FF$. Let $\eps=\frac{\kappa}{4MT}$. Let $t_0>0$ be such that $\|\delta(t)\|_\FF\leq\eps$ for all $t\geq t_0$. Then for all $t\geq t_0$ and all $x\in\FF$,
\begin{align*}
    &\int_t^{t+T}|\langle g(\tau)+\delta(\tau), x\rangle_\FF|^2\dd\tau
\geq
\int_t^{t+T}|\langle g(\tau), x\rangle_\FF|^2\dd\tau
\\
&\quad
-2 \int_t^{t+T}\left|\langle g(\tau), x\rangle_\FF \langle \delta(\tau), x\rangle_\FF\right|\dd\tau
\\
&\geq (\kappa - 2MT\eps)\|x\|_\FF^2
\\
&\geq \frac{\kappa}{2} \|x\|_\FF^2.
\end{align*}
Hence, $g+\delta$ is PE with respect to $\Id_\FF$ by \ref{i1}.
\item This proof follows the one given in \cite[Property 4]{1393135}. We give it here for the sake of completeness. Denote by $M$ a bound of $g$ and $\dot g$. Let $z$ be a solution of $\frac{\dd z}{\dd t} = -\mu z(t) + g(t)$. By Duhamel's formula,
    \begin{align*}\|z(t)\|_\FF &\leq e^{-\mu t}\|z(0)\|_\FF + \int_0^te^{-\mu (t-\tau)}\|g(\tau)\|_\FF\dd \tau \\
    &\leq e^{-\mu t}\|z(0)\|_\FF + \frac{M}{\mu}.
    \end{align*}
Hence there exists $t_0\geq 0$ such that $\|z(t)\|_\FF\leq \frac{2M}{\mu}$ for all $t\geq t_0$.
For any $x\in \FF$, define $\phi_x:\R_+\to\R$ by $\phi_x(t) = - \langle z(t), x\rangle_\FF\langle g(t), x\rangle_\FF$. Then $\phi_x$ is continuously differentiable and $\dot \phi_x = \langle z, x\rangle_\FF\langle \mu g - \dot g, x\rangle_\FF - |\langle g, x\rangle_\FF|^2$. Hence, for all $t\geq t_0$ and all $T>0$,
\begin{align}\label{eq:phi}
    \phi_x(t+T)-\phi_x(t) =&\, \int_t^{t+T}\langle z(\tau), x\rangle_\FF\langle \mu g(\tau) - \nonumber\\
    \dot g(\tau), x\rangle_\FF\dd \tau&\,- \int_t^{t+T}|\langle g(\tau), x\rangle_\FF|^2\dd \tau.
\end{align}
Since $g$ is PE with respect to $\Id_\FF$, there exist $T, \kappa>0$ such that, for any $k\in\N$,
\begin{equation}\label{eq:k}
    \int_t^{t+kT}|\langle g(\tau), x\rangle_\FF|^2\dd\tau \geq k\kappa\|x\|_\FF^2,\quad \forall x\in \FF,\ \forall t\geq0.
\end{equation}
Moreover, by Cauchy-Schwartz inequality, $|\phi_x(t)|\leq \frac{2M^2}{\mu}\|x\|^2$. Hence $\phi_x(t+T)-\phi_x(t) \geq -\frac{4M^2}{\mu}\|x\|^2$. Choose $k$ large enough that $k\kappa > \frac{4M^2}{\mu}$. Combining \eqref{eq:phi} and \eqref{eq:k} yields, for all $t\geq t_0$ and all $x\in \FF$,
\small
\begin{align}\label{eq:1}
    \int_t^{t+kT}\langle z(\tau), x\rangle_\FF\langle \mu g(\tau) - \dot g(\tau), x\rangle_\FF\dd \tau
    \geq \left(k\kappa -\frac{4M^2}{\mu}\right)\|x\|_\FF^2.
\end{align}
\normalsize
Finally, we get by Cauchy-Schwartz inequality that
\begin{align}
\int_t^{t+kT}\langle &z(\tau), x\rangle_\FF\langle \mu g(\tau) - \dot g(\tau), x\rangle_\FF\dd \tau
\leq (\mu+1) M\nonumber\\
&\|x\|_\FF\int_t^{t+kT}\langle z(\tau), x\rangle_\FF\dd \tau\label{eq:2}
\end{align}
and
\begin{align}\label{eq:3}
\int_t^{t+kT} \langle z(\tau), x\rangle_\FF\dd \tau \leq \sqrt{kT}\sqrt{\int_t^{t+kT} |\langle z(\tau), x\rangle_\FF|^2\dd \tau}.
\end{align}
Combining \eqref{eq:1}-\eqref{eq:2}-\eqref{eq:3}, we obtain that for all $x\in \FF$ and all $t\geq t_0$,
\begin{align*}
\int_t^{t+kT} |\langle z(\tau), x\rangle_\FF|^2\dd \tau \geq 
\frac{(k\kappa -\frac{4M^2}{\mu})^2}{kT(\mu +1)^2M^2}\|x\|_\FF^2
\end{align*}
which implies that $z$ is persistently exciting with respect to $\Id_\FF$ by Lemma~\ref{lem:PE0} and \ref{i1}.
\end{enumerate}

\bibliographystyle{abbrv}
\bibliography{references}

\begin{thebibliography}{10}

\bibitem{alswaihli2018kernel}
J.~Alswaihli, R.~Potthast, I.~Bojak, D.~Saddy, and A.~Hutt.
\newblock Kernel reconstruction for delayed neural field equations.
\newblock {\em The Journal of Mathematical Neuroscience}, 8(1):1--24, 2018.

\bibitem{Atay:2004bg}
F.~M. Atay and A.~Hutt.
\newblock {Stability and Bifurcations in Neural Fields with Finite Propagation
  Speed and General Connectivity}.
\newblock {\em SIAM Journal on Applied Mathematics}, 65(2):644--666, Jan. 2004.

\bibitem{Bertalmio:2021uo}
M.~Bertalm{\'\i}o, L.~Calatroni, V.~Franceschi, B.~Franceschiello, and
  D.~Prandi.
\newblock {Cortical-Inspired Wilson{\textendash}Cowan-Type Equations for
  Orientation-Dependent Contrast Perception Modelling}.
\newblock {\em Journal of Mathematical Imaging and Vision}, 63(2):263--281,
  2021.

\bibitem{BESANCON2000271}
G.~Besançon.
\newblock Remarks on nonlinear adaptive observer design.
\newblock {\em Systems \& Control Letters}, 41(4):271--280, 2000.

\bibitem{BESANCON201715416}
G.~Besançon and A.~Ţiclea.
\newblock On adaptive observers for systems with state and parameter
  nonlinearities.
\newblock {\em IFAC-PapersOnLine}, 50(1):15416--15421, 2017.
\newblock 20th IFAC World Congress.

\bibitem{Boscain:2021we}
U.~Boscain, D.~Prandi, L.~Sacchelli, and G.~Turco.
\newblock {A bio-inspired geometric model for sound reconstruction}.
\newblock {\em The Journal of Mathematical Neuroscience}, 11(1):2, 2021.

\bibitem{Bressloff2011}
P.~Bressloff.
\newblock {Spatiotemporal dynamics of continuum neural fields}.
\newblock {\em Journal of Physics A: Mathematical and Theoretical}, 45(3),
  2012.

\bibitem{git}
L.~Brivadis.
\newblock {K}ernel{E}stimation {P}roject.
\newblock https://github.com/brivadis/KernelEstimation, 2021.

\bibitem{brivadis:hal-03660185}
L.~Brivadis, A.~Chaillet, and J.~Auriol.
\newblock Online estimation of hilbert-schmidt operators and application to
  kernel reconstruction of neural fields.
\newblock In {\em 2022 IEEE 61st Conference on Decision and Control (CDC)},
  pages 597--602, 2022.

\bibitem{brivadis:hal-03589737}
L.~Brivadis, C.~Tamekue, A.~Chaillet, and J.~Auriol.
\newblock Existence of an equilibrium for delayed neural fields under output
  proportional feedback.
\newblock {\em Automatica}, 151:110909, 2023.

\bibitem{https://doi.org/10.48550/arxiv.2111.02176}
T.~B. Burghi and R.~Sepulchre.
\newblock Online estimation of biophysical neural networks, 2021.

\bibitem{Carlu2020}
M.~Carlu, O.~Chehab, L.~Dalla~Porta, D.~Depannemaecker, C.~H\'{e}ric\'{e},
  M.~Jedynak, E.~K\"{o}ksal~Ers\"{o}z, P.~Muratore, S.~Souihel, C.~Capone,
  Y.~Zerlaut, A.~Destexhe, and M.~di~Volo.
\newblock A mean-field approach to the dynamics of networks of complex neurons,
  from nonlinear integrate-and-fire to hodgkin–huxley models.
\newblock {\em Journal of Neurophysiology}, 123(3):1042--1051, 2020.
\newblock PMID: 31851573.

\bibitem{Carron:2013hi}
R.~Carron, A.~Chaillet, A.~Filipchuk, W.~Pasillas-L{\'e}pine, and C.~Hammond.
\newblock {Closing the loop of deep brain stimulation.}
\newblock {\em Frontiers in Systems Neuroscience}, 7:112, 2013.

\bibitem{DECH16}
A.~Chaillet, G.~Detorakis, S.~Palfi, and S.~Senova.
\newblock {Robust stabilization of delayed neural fields with partial
  measurement and actuation}.
\newblock {\em Automatica}, 83:262--274, Sep. 2017.

\bibitem{coombes2014neural}
S.~Coombes, P.~beim Graben, R.~Potthast, and J.~Wright.
\newblock {\em Neural Fields: Theory and Applications}.
\newblock Springer, 2014.

\bibitem{761927}
R.~Curtain, M.~Demetriou, and K.~Ito.
\newblock Adaptive observers for slowly time varying infinite dimensional
  systems.
\newblock In {\em Proceedings of the 37th IEEE Conference on Decision and
  Control (Cat. No.98CH36171)}, volume~4, pages 4022--4027 vol.4, 1998.

\bibitem{DEMETRIOU19965346}
M.~A. Demetriou and K.~Ito.
\newblock Adaptive observers for a class of infinite dimensional systems.
\newblock {\em IFAC Proceedings Volumes}, 29(1):5346--5350, 1996.
\newblock 13th World Congress of IFAC, 1996, San Francisco USA, 30 June - 5
  July.

\bibitem{DECHPASE15}
G.~Detorakis, A.~Chaillet, S.~Palfi, and S.~Senova.
\newblock Closed-loop stimulation of a delayed neural fields model of
  parkinsonian {STN-GPe} network: a theoretical and computational study.
\newblock {\em Frontiers in Neuroscience}, 9(237), 2015.

\bibitem{DECHcdc17}
G.~I. Detorakis and A.~Chaillet.
\newblock Incremental stability of spatiotemporal delayed dynamics and
  application to neural fields.
\newblock In {\em 56th IEEE Conference on Decision and Control}, pages
  5937--5942, 2017.

\bibitem{Detorakis2015}
G.~I. Detorakis, A.~Chaillet, S.~Palfi, and S.~Senova.
\newblock Closed-loop stimulation of a delayed neural fields model of
  parkinsonian {STN}-{GPe} network: a theoretical and computational study.
\newblock {\em Frontiers in Neuroscience}, 9, July 2015.

\bibitem{Detorakis:2014km}
G.~I. Detorakis and N.~P. Rougier.
\newblock {Structure of receptive fields in a computational model of area 3b of
  primary sensory cortex}.
\newblock {\em Frontiers in computational neuroscience}, 8, 2014.

\bibitem{ermentrout1993existence}
G.~B. Ermentrout and J.~B. McLeod.
\newblock Existence and uniqueness of travelling waves for a neural network.
\newblock {\em Proceedings of the Royal Society of Edinburgh Section A:
  Mathematics}, 123(3):461--478, 1993.

\bibitem{FARZA20092292}
M.~Farza, M.~M’Saad, T.~Maatoug, and M.~Kamoun.
\newblock Adaptive observers for nonlinearly parameterized class of nonlinear
  systems.
\newblock {\em Automatica}, 45(10):2292--2299, 2009.

\bibitem{Faugeras:2008wx}
O.~Faugeras, F.~Grimbert, and J.-J. Slotine.
\newblock {Absolute stability and complete synchronization in a class of neural
  fields models}.
\newblock {\em SIAM Journal of Applied Mathematics}, 61(1):205--250, 2008.

\bibitem{Faugeras:2009gz}
O.~Faugeras, R.~Veltz, and F.~Grimbert.
\newblock {Persistent neural states: stationary localized activity patterns in
  nonlinear continuous n-population, q-dimensional neural networks.}
\newblock {\em Neural Computation}, 21(1):147--187, Jan. 2009.

\bibitem{FAYE2010561}
G.~Faye and O.~Faugeras.
\newblock Some theoretical and numerical results for delayed neural field
  equations.
\newblock {\em Physica D: Nonlinear Phenomena}, 239(9):561--578, 2010.
\newblock Mathematical Neuroscience.

\bibitem{Gohberg1990}
I.~Gohberg, S.~Goldberg, and M.~A. Kaashoek.
\newblock {\em Hilbert-Schmidt Operators}, pages 138--147.
\newblock Birkh{\"a}user Basel, Basel, 1990.

\bibitem{hale2013introduction}
J.~K. Hale and S.~M.~V. Lunel.
\newblock {\em Introduction to functional differential equations}, volume~99.
\newblock Springer Science \& Business Media, 2013.

\bibitem{Holgado2010}
A.~J.~N. Holgado, J.~R. Terry, and R.~Bogacz.
\newblock Conditions for the generation of beta oscillations in the subthalamic
  nucleus-globus pallidus network.
\newblock {\em Journal of Neuroscience}, 30(37):12340--12352, Sept. 2010.

\bibitem{IOANNOU1984583}
P.~Ioannou and P.~Kokotovic.
\newblock Instability analysis and improvement of robustness of adaptive
  control.
\newblock {\em Automatica}, 20(5):583--594, 1984.

\bibitem{Laing:2002we}
C.~R. Laing, W.~C. Troy, B.~Gutkin, and G.~Ermentrout.
\newblock {Multiple bumps in a neuronal model of working memory}.
\newblock {\em SIAM Journal on Applied {\ldots}}, 2002.

\bibitem{Limousin:1998cc}
P.~Limousin, P.~Krack, P.~Pollak, A.~Benazzouz, C.~Ardouin, D.~Hoffmann, and
  A.~L. Benabid.
\newblock {Electrical stimulation of the subthalamic nucleus in advanced
  Parkinson's disease}.
\newblock {\em N Engl J Med}, 339(16):1105--1111, Jan. 1998.

\bibitem{1393135}
A.~Loria, E.~Panteley, D.~Popovic, and A.~Teel.
\newblock A nested matrosov theorem and persistency of excitation for uniform
  convergence in stable nonautonomous systems.
\newblock {\em IEEE Transactions on Automatic Control}, 50(2):183--198, 2005.

\bibitem{mironchenko2023input}
A.~Mironchenko.
\newblock Input-to-state stability.
\newblock In {\em Input-to-State Stability: Theory and Applications}, pages
  41--115. Springer, 2023.

\bibitem{Mironchenko2023}
A.~Mironchenko.
\newblock {\em Input-to-State Stability: Theory and Applications}.
\newblock Springer International Publishing, 2023.

\bibitem{doi:10.1080/00207178708933715}
K.~S. Narendra and A.~M. Annaswamy.
\newblock Persistent excitation in adaptive systems.
\newblock {\em International Journal of Control}, 45(1):127--160, 1987.

\bibitem{PINOTSIS2014143}
D.~Pinotsis, N.~Brunet, A.~Bastos, C.~Bosman, V.~Litvak, P.~Fries, and
  K.~Friston.
\newblock Contrast gain control and horizontal interactions in v1: A dcm study.
\newblock {\em NeuroImage}, 92:143--155, 2014.

\bibitem{https://doi.org/10.48550/arxiv.2112.05497}
A.~Pyrkin, A.~Bobtsov, R.~Ortega, and A.~Isidori.
\newblock An adaptive observer for uncertain linear time-varying systems with
  unknown additive perturbations, 2021.

\bibitem{sastry1990adaptive}
S.~Sastry and M.~Bodson.
\newblock Adaptive control: stability, convergence, and robustness, 1990.

\bibitem{Sontag2008}
E.~D. Sontag.
\newblock {\em Input to State Stability: Basic Concepts and Results}, pages
  163--220.
\newblock Springer Berlin Heidelberg, Berlin, Heidelberg, 2008.

\bibitem{tamekue2023mathematical}
C.~Tamekue, D.~Prandi, and Y.~Chitour.
\newblock On the mathematical replication of the mackay effect from redundant
  stimulation.
\newblock {\em arXiv preprint arXiv:2311.07338}, 2023.

\bibitem{temam1986infinite}
R.~Temam.
\newblock Infinite-dimensional dynamical systems.
\newblock {\em Nonlinear functional analysis and its applications, Part 2},
  45(Part 2):431, 1986.

\bibitem{Veltz:2013wm}
R.~Veltz.
\newblock {Interplay between synaptic delays and propagation delays in neural
  field equations}.
\newblock {\em SIAM Journal on Applied Dynamical Systems}, 2013.

\bibitem{8786148}
J.~Wang, D.~Efimov, and A.~A. Bobtsov.
\newblock On robust parameter estimation in finite-time without persistence of
  excitation.
\newblock {\em IEEE Transactions on Automatic Control}, 65(4):1731--1738, 2020.

\bibitem{wilson1973mathematical}
H.~R. Wilson and J.~D. Cowan.
\newblock A mathematical theory of the functional dynamics of cortical and
  thalamic nervous tissue.
\newblock {\em Kybernetik}, 13(2):55--80, 1973.

\end{thebibliography}

\end{document}